\documentclass[AMA,Times1COL]{WileyNJDv5} 

\usepackage{amssymb}
\usepackage{subcaption}
\usepackage{enumitem}
\usepackage{amsopn}
\DeclareMathOperator*{\argmax}{arg\,max}
\DeclareMathOperator*{\argmin}{arg\,min}
\DeclareMathOperator*{\rank}{rank}
\DeclareMathOperator*{\sign}{sign}

\DeclareMathOperator*{\diag}{diag}

\newcommand{\eq}[1]{(\ref{eq:#1})}
\newcommand{\lem}[1]{Lemma~\ref{lemma:#1}}

\newcommand{\thm}[1]{Theorem~\ref{theorem:#1}}
\newcommand{\dfn}[1]{Definition~\ref{definition:#1}}
\newcommand{\sect}[1]{Section~\ref{sec:#1}}
\newcommand{\fig}[1]{Figure~\ref{fig:#1}}
\newcommand{\alg}[1]{Algorithm~\ref{alg:#1}}
\newcommand{\tab}[1]{Table~\ref{table:#1}}

\articletype{Research article}%

\received{Date Month Year}
\revised{Date Month Year}
\accepted{Date Month Year}
\journal{Numerical Linear Algebra with Applications}
\volume{00}
\copyyear{2025}
\startpage{1}

\begin{document}

\title{Accelerated alternating minimization algorithm for low-rank approximations in the Chebyshev norm}

\author[1,2]{Stanislav Morozov}

\author[1,3]{Dmitry Zheltkov}

\author[4,1]{Alexander Osinsky}

\authormark{MOROZOV \textsc{et al.}}
\titlemark{Accelerated alternating minimization algorithm for low-rank approximations in the Chebyshev norm}

\address[1]{\orgname{Marchuk Institute of Numerical Mathematics of the Russian Academy of Sciences}, \orgaddress{Gubkin Street 8, Moscow, 119333, Russia}}

\address[2]{\orgname{National Research University Higher School of Economics}, \orgaddress{Pokrovsky blvd., 11, Moscow, 109028, Russia}}

\address[3]{\orgname{Lomonosov Moscow State University}, \orgaddress{GSP-1, Leninskie Gory, Moscow, 119991, Russia}}

\address[4]{\orgname{Skolkovo Institute of Science and Technology}, \orgaddress{Bolshoy blvd., 30,  bld. 1, Moscow, 121205, Russia}}

\corres{Corresponding author Stanislav Morozov, \email{stanis-morozov@yandex.ru}}


\fundingInfo{Russian Science Foundation, Project 25-71-00096 (the project webpage is available \url{https://rscf.ru/project/25-71-00096/}).}

\abstract[Abstract]{Nowadays, low-rank approximations of matrices are an important component of many methods in science and engineering. Traditionally, low-rank approximations are considered in unitary invariant norms, however, recently element-wise approximations have also received significant attention in the literature. In this paper, we propose an accelerated alternating minimization algorithm for solving the problem of low-rank approximation of matrices in the Chebyshev norm. Through the numerical evaluation we demonstrate the effectiveness of the proposed procedure for large-scale problems. We also theoretically investigate the alternating minimization method and introduce the notion of a $2$-way alternance of rank $r$. We show that the presence of a $2$-way alternance of rank $r$ is the necessary condition of the optimal low-rank approximation in the Chebyshev norm and that all limit points of the alternating minimization method satisfy this condition.}


\jnlcitation{\cname{%
\author{Morozov S.},
\author{Zheltkov D.}, and
\author{Osinsky A}}.
\ctitle{Accelerated alternating minimization algorithm for low-rank approximations in the Chebyshev norm}, \cjournal{\it Numerical Linear Algebra with Applications}, \cvol{2025;00(00):1--18}.}

\maketitle

%

\section{Introduction}
\label{sec:intro}
Low-rank matrix approximation algorithms are a crucial component in different fields, such as differential equations \cite{matveev2015fast}, computational fluid dynamics \cite{son2014data}, recommender systems \cite{he2016fast} and machine learning \cite{sainath2013low}. Nowadays, most of the methods tackle the problem of low-rank matrix approximation in the unitary invariant norms. For such norms there exist efficient algorithms, e.g. SVD, that provide the optimal approximation. The quality of the approximation for the unitary invariant norms is related to the decay rate of singular values. Only matrices with fast decay of singular values can be reasonably approximated in these norms. However, in some applications, matrices arise that can be successfully approximated by low-rank structures in other norms, independently of the singular values decay rate \cite{udell2019big}.

In this paper, we address the problem of building low-rank approximations in the Chebyshev norm. Let $A \in \mathbb{R}^{m \times n}$ and $r \ge 1$ be an integer. Then the problem of rank-$r$ approximation reads as
\begin{equation}
\label{eq:main_problem}
    \|A - UV^T\|_C \to \min\limits_{U \in \mathbb{R}^{m\times r}, V \in \mathbb{R}^{n \times r}},
\end{equation}
where $\|X\|_C = \max\limits_{i,j}|x_{ij}|$. To solve the problem \eq{main_problem}, in \cite{zamarashkin2022best} the authors propose the \textit{alternating minimization method}. The method alternately solves the problem
\begin{equation}
\label{eq:one_matrix_problem}
    \|A - UV^T\|_C \to \min\limits_{U \in \mathbb{R}^{m\times r}}
\end{equation}
for a fixed $V \in \mathbb{R}^{n \times r}$ and
\begin{equation*}
    \|A - UV^T\|_C \to \min\limits_{V \in \mathbb{R}^{n\times r}}
\end{equation*}
for a fixed $U \in \mathbb{R}^{m \times r}$. One can see that \eq{one_matrix_problem} can be decomposed to the set of problems of the \textit{best uniform approximation}.
\begin{equation}
\label{eq:uniform_approximation_problem}
    \|a - Vu\|_\infty \to \min\limits_{u \in \mathbb{R}^{r}}.
\end{equation}
To the best of our knowledge, the first algorithm for solving \eq{uniform_approximation_problem} is proposed in \cite{zamarashkin2022best}.

The main contribution of this paper is as follows. We propose the accelerated algorithm for solving \eq{uniform_approximation_problem}, which is equivalent to the one in \cite{zamarashkin2022best} in precise arithmetic, but has lower complexity. We also propose the equioscillation criterion for a vector to be the solution of \eq{uniform_approximation_problem}, which is similar to the famous Chebyshev equioscillation theorem for the approximation of continuous functions by polynomials. Moreover, we analyze the alternating minimization method for solving \eq{main_problem} and propose the notion of a \textit{2-way alternance of rank $r$}, which generalizes the similar concept proposed in \cite{daugavet1971uniform} for rank-$1$ approximations. We extend this notion to the arbitrary rank and demonstrate that the presence of a 2-way alternance of rank-$r$ is the necessary condition for a pair of matrices $(U, V)$ to be the solution of \eq{main_problem} and that all limit points of the alternating minimization method possess this structure. Finally, we conduct the extensive numerical evaluation for constructing low-rank approximations in the Chebyshev norm on large-scale matrices. We also compare the results with the alternating projections method \cite{budzinskiy2023quasioptimal}, another approach for solving \eq{main_problem}.

The rest of the paper is organized as follows. In \sect{related} we present an overview of recent results in the field of low-rank Chebyshev approximations. \sect{prelim} contains the basic results connected with solving the problem of the best uniform approximation \eq{uniform_approximation_problem}. In \sect{equiosc} we present the equioscillation criterion for a vector to be the solution of the problem \eq{uniform_approximation_problem}. In \sect{accelerated_alg} we describe the accelerated uniform approximation algorithm. \sect{alternating_minim_description} contains the formal description of the alternating minimization method and in \sect{rank_r_alternance} we provide the notion of a 2-way alternance of rank-$r$ and demonstrate that the presence of the alternance is the necessary condition of the optimal solution to \eq{main_problem}, as well as all limit points of the alternating minimization method fulfill this condition. Finally, \sect{numerical} contains the numerical evaluation of the alternating minimization method for solving \eq{main_problem} and \sect{conclusion} concludes the paper.

\section{Related work}
\label{sec:related}
A matrix admits efficient low-rank approximation in unitary invariant norms if it has fast decay of singular values. Otherwise, there is no reasonable low-rank approximation in such norms. The situation turns out to be different if we consider approximations in the Chebyshev norm. The seminal result from \cite{udell2019big} states that for any matrix $X \in \mathbb{R}^{m \times n}$, where $m \ge n$, and any $\varepsilon > 0$ there exists a matrix $Y \in \mathbb{R}^{m \times n}$ of rank $r$, where
\begin{equation}
\label{eq:rank_est}
    r \le \lceil 72\log{(2n + 1)}/\varepsilon^2 \rceil,
\end{equation}
such that $\|X - Y\|_C \le \varepsilon \|X\|_2$. It means that for any fixed $\varepsilon > 0$ and a sequence of matrices $\{X_n\}_n$ with the bounded spectral norm and growing size $n$, there exists a sequence of matrices $\{Y_n\}_n$ of rank $O(\log n)$, such that $\|X_n - Y_n\|_C \le \varepsilon$. Note that at least the constant in \eq{rank_est} is overestimated (see also \cite{budzinskiy2024distance} for sharper estimates with the notion of $\mu$-coherence). For example, we demonstrate in \sect{numerical} that the identity matrix of size $16,384$ can be approximated with accuracy $0.1$ by a matrix of rank $333$.

The problem of constructing low-rank approximations in the Chebyshev norm is challenging. In \cite{gillis2019low} it is shown that even for rank-1 approximations the problem of checking whether for a matrix $A \in \mathbb{R}^{m\times n}$ and a number $\varepsilon > 0$ there exist vectors $u \in \mathbb{R}^m$ and $v \in \mathbb{R}^n$ such that $\|A - uv^T\|_C < \varepsilon$, is NP-complete.

To the best of our knowledge, the first algorithm for solving the problem of low-rank Chebyshev approximation was proposed in \cite{daugavet1971uniform} for the special case of rank-$1$ approximations. Let $A \in \mathbb{R}^{m \times n}$ be a matrix to be approximated and $v^{(0)} \in \mathbb{R}^n$. Then the author in \cite{daugavet1971uniform} uses the \textit{alternating minimization method}, which solves the problems
\begin{equation*}
    u^{(t)} \leftarrow \argmin\limits_{u \in \mathbb{R}^m} \|A - u v^{(t)}\|_C, \quad v^{(t+1)} \leftarrow \argmin\limits_{v \in \mathbb{R}^n} \|A - u^{(t)} v\|_C
\end{equation*}
for $t = 0, 1, 2, \dots$. Also \cite{daugavet1971uniform} introduces the notion of \textit{$2$-way alternance} and demonstrate that all limit points of the sequence $\{(u^{(t)}, v^{(t)})\}_t$ possess this alternance. It is also shown in \cite{daugavet1971uniform} that if there is a $2$-way alternance for a matrix $A \in \mathbb{R}^{m\times n}$ and vectors $\hat{u} \in \mathbb{R}^m$ and $\hat{v} \in \mathbb{R}^n$, then for any $u \in \mathbb{R}^m$ such that $\sign u = \sign \hat{u}$ we have $\|A - \hat{u}\hat{v}^T\|_C \le \|A - uv^T\|_C$ for any $v \in \mathbb{R}^n$. The similar property is true for any $v \in \mathbb{R}^n$ such that $\sign v = \sign \hat{v}$. This results was used in \cite{morozov2023optimal} to derive a method that is capable of constructing \textit{optimal} rank-1 Chebyshev approximations.

In \cite{zamarashkin2022best} the authors proposed a method for solving the problem
\begin{equation}
\label{eq:uniform_approximation_problem2}
    \|a - Vu\|_\infty \to \min\limits_{u \in \mathbb{R}^{r}}
\end{equation}
and generalized the alternating minimization method to arbitrary rank approximations. In this paper, we propose a new method for solving \eq{uniform_approximation_problem2}, which has the lower complexity than the one proposed in \cite{zamarashkin2022best}. We also extend the notion of a $2$-way alternance to arbitrary rank and demonstrate that the presence of this structure is the necessary condition of the optimal approximation and that all limit points of the alternating minimization method satisfy this condition.

Another line of research is related to the \textit{alternating projections} method \cite{budzinskiy2023quasioptimal, budzinskiy2024distance, budzinskiy2024entrywise, budzinskiy2024big}. The method alternately projects to the set of low-rank matrices via SVD and to the $\varepsilon$--ball in the Chebyshev norm with the center being the target matrix. The value $\varepsilon$ is found using binary search. In \sect{numerical} we compare alternating minimization and alternating projections methods for building low-rank Chebyshev approximations.

\section{Preliminaries}
\label{sec:prelim}
Let $V \in \mathbb{R}^{n \times r}$, where $n \ge r$ and $a \in \mathbb{R}^n$. In this section, we provide the fundamental properties of the problem
\begin{equation}
\label{eq:uniform_approx_problem}
\|Vu - a\|_\infty \to \min\limits_{u \in \mathbb{R}^r}.
\end{equation}
Here we present the required theory without proofs, for further details see \cite{zamarashkin2022best}.

Let $S$ be an ordered set of integers $1 \le i_1, i_2, \dots, i_k \le n$. Here and further we denote with parentheses an ordered set, e.g. $S=(i_1, i_2, \dots, i_k)$. Let us denote by $V(S)$ the submatrix of matrix $V$, containing the rows with the numbers from the set $S$. If $n = r + 1$, we denote by $V^{\setminus j}$ the submatrix of matrix $V$ containing all rows except row number $j$, that is, $V^{\setminus j} = V((1, \dots, j-1, j+1, \dots, r+1))$. For matrices of size $(r+1)\times r$ we also denote by $D_j(V) = \det{V^{\setminus j}}$. Similarly, if $a \in \mathbb{R}^n$, we denote by $a(S)$ the subvector of vector $a$, containing the elements with the numbers from the set $S$ and $a_{\setminus j}$ denotes $a((1, \dots, j-1, j+1, \dots, n))$.

The key concept associated with solving the problem \eq{uniform_approx_problem} is the notion of the \textit{Chebyshev matrix}.
The solution to the problem \eq{uniform_approx_problem} is not always unique. The necessary and sufficient condition for \eq{uniform_approx_problem} to have the unique solution for any right-hand side $a$ can be provided by the following (see \cite[Theorem 2, Theorem 3]{zamarashkin2022best}).
\begin{definition}
\label{definition:chebyshev_matrix}
A matrix $V \in \mathbb{R}^{n \times r}$ with $n \ge r$ is called \textit{Chebyshev} if all its $r \times r$ submatrices are non-singular.
\end{definition}
\begin{theorem}
\label{theorem:exists_unique_cont}
    Let $V \in \mathbb{R}^{n \times r}$ be a Chebyshev matrix and $a \in \mathbb{R}^n$. Then the solution to the problem \eq{uniform_approx_problem} exists, is unique and continuously depends on the matrix $V$ and right-hand side $a$.
\end{theorem}
The solution to the problem possesses the following property (see \cite[Proposition 2]{zamarashkin2022best}).
\begin{lemma}
\label{lemma:rp1points}
Let a matrix $V \in \mathbb{R}^{n \times r}$, $n > r$, be Chebyshev, $a \in \mathbb{R}^n$, $\hat{u} \in \mathbb{R}^r$ is the solution of \eq{uniform_approx_problem}. Let $w = a - V\hat{u}$. Then there are at least $r+1$ distinct integers $1 \le i_1 < \dots < i_{r+1} \le n$ such that
\begin{equation*}
    |w_{i_j}| = \|w\|_\infty, ~~~ j = 1, \dots, r + 1.
\end{equation*}
\end{lemma}
Let $J = (1, 2, \dots, n)$ and $J' \subset J$. Let
\begin{equation*}
	\mu(J') = \min\limits_{u \in \mathbb{R}^{r}} \left\| a(J') - V(J') u \right\|_{\infty}.
\end{equation*}
The following definition reveals many important properties connected with the problem \eq{uniform_approx_problem}.
\begin{definition}
\label{definition:charact_set}
A set $J'$ is called a \textit{characteristic set} if $\mu(J) = \mu(J')$ and for any subset $J'' \subsetneq J'$ $\mu(J'') < \mu(J)$ holds.
\end{definition}
The proof of the following theorem can be found in \cite[Theorem 6]{zamarashkin2022best}.
\begin{theorem}
\label{theorem:charact_set_size}
Let $V \in \mathbb{R}^{n \times r}$, where $n > r$ and $a \in \mathbb{R}^n$ does not belong to the image of $V$. Then there exists at least one characteristic set consisting of at most $r + 1$ elements. Moreover, if the matrix $V$ is Chebyshev, then any characteristic set consists of at least $r + 1$ elements.
\end{theorem}

The next theorem provides a useful criterion for a vector to be the solution to the problem \eq{uniform_approx_problem} (see \cite[Theorem 10]{zamarashkin2022best}).
\begin{theorem}
\label{theorem:remez_opt_crit}
Let $V \in \mathbb{R}^{n \times r}$ and $a \in \mathbb{R}^n$. Let $\hat{u} \in \mathbb{R}^r$ and let us denote $w = a - V\hat{u}$ and
\begin{equation*}
	J = \{j \in \{1, 2, \dots, n\} : |w_j| = \|w\|_\infty \}.
\end{equation*}
Then $\hat{u}$ is the solution to the problem \eq{uniform_approx_problem} if and only if there is a non-zero vector $\delta \in \mathbb{R}^{|J|}$ with non-negative components such that
\begin{equation*}
    V(J)^T \diag{(\sign w(J))} \delta = 0.
\end{equation*}
\end{theorem}

Due to \thm{remez_opt_crit} the case $n = r+1$ is particularly important.
The following theorem provides a direct formula for this problem (see \cite{dzyadyk1974uniform, dzyadyk1977book}).
\begin{theorem}
\label{theorem:dzyadyk_sol}
    Let $V \in \mathbb{R}^{(r+1) \times r}$ be a Chebyshev matrix and $a \in \mathbb{R}^{r+1}$ does not belong to the image of $V$. Let us denote $\hat{u}^{(j)} = (V^{\setminus j})^{-1} a_{\setminus j}$. Then
    \begin{enumerate}[label=(\roman*)]
        \item\label{dzyadyk_sol_i} It holds
        \begin{equation*}
            (\hat{u}^{(j)})^T v^j - a_j = \dfrac{(-1)^{j+1}}{D_j(V)} \sum\limits_{k=1}^{r+1} (-1)^{k} a_{k} D_{k}(V),
        \end{equation*}
        where $v^j$ is $j$--th row of the matrix $V$.
        \item\label{dzyadyk_sol_ii} The vector
        \begin{equation*}
            \hat{u} = \sum\limits_{j=1}^{r+1} \dfrac{|D_j(V)|}{\sum\limits_{k=1}^{r+1} |D_k(V)|} \hat{u}^{(j)}
        \end{equation*}
        is the solution to the problem \eq{uniform_approx_problem} and
        \begin{equation*}
            \min\limits_{u \in \mathbb{R}^{r}} \left\| a - V u \right\|_{\infty} = \dfrac{\left| \sum\limits_{j=1}^{r+1} (-1)^j D_j(V) a_j \right|}{\sum\limits_{j=1}^{r+1} |D_j(V)|}.
        \end{equation*}
    \end{enumerate}
\end{theorem}

We will also require the following simple result on the properties of orthogonal matrices.
\begin{lemma}
\label{lemma:orthog_complement}
    Let $\hat{Q} \in \mathbb{R}^{n \times n}$ be an orthogonal matrix of the form $\hat{Q} = \begin{bmatrix}
        Q & q'
    \end{bmatrix}$, where $Q \in \mathbb{R}^{n \times (n-1)}$ and $q' \in \mathbb{R}^n$. Then
    \begin{equation*}
        q' = (-1)^n \sign\det\hat{Q} \begin{bmatrix}
            (-1) D_1(Q) & (-1)^2 D_2(Q) & \dots & (-1)^n D_n(Q)
        \end{bmatrix}^T.
    \end{equation*}
\begin{proof}
    Let us consider the system $\hat{Q} y^{(k)} = e_k$, where $e_k$ is the standard basis vector. Then $y^{(k)}_n = (\hat{Q}^T e_k)_n = q'_k$. However, from Cramer's rule
    \begin{equation*}
        y^{(k)}_n = \dfrac{\det \hat{Q}^{(k)}}{\det \hat{Q}},
    \end{equation*}
    where $\hat{Q}^{(k)}$ is the matrix formed by replacing the last column of $\hat{Q}$ by the vector $e_k$. Clearly, $\det \hat{Q}^{(k)} = (-1)^{n + k} D_j(Q)$.
\end{proof}
\end{lemma}

The next lemma provides a convenient way of computing the best uniform approximation error, that we will need in our accelerated algorithm.
\begin{lemma}
\label{lemma:orthog_solution_formula}
    Let $V \in \mathbb{R}^{(r+1) \times r}$ be a Chebyshev matrix and $a \in \mathbb{R}^n$. Let $q \in \mathbb{R}^{r+1}$ be a non-zero vector such that $V^T q = 0$. Then
    \begin{equation}
    \label{eq:min_value}
        \min\limits_{u \in \mathbb{R}^{r}} \left\| a - V u \right\|_{\infty} = \dfrac{|q^T a|}{\|q\|_1}.
    \end{equation}
\begin{proof}
    Follows directly from \thm{dzyadyk_sol}~\ref{dzyadyk_sol_ii} and \lem{orthog_complement}.
\end{proof}
\end{lemma}

\section{Equioscillation theorem}
\label{sec:equiosc}
One of the seminal results on the uniform approximation of continuous functions using polynomials is the Chebyshev equioscillation theorem (see e.g. \cite{dzyadyk1977book}).
\begin{theorem}
\label{theorem:classical_equioscillation_theorem}
Let $f$ be a continuous function from $[ a , b ]$ to $\mathbb{R}$. Among all the polynomials of degree not greater than $d$, the polynomial $g$ minimizes the uniform norm of the difference $\|f - g\|_C$ if and only if there are $d + 2$ points $a \le x_0 < x_1 < \dots < x_{d+1} \le b$ such that
\begin{equation*}
    f(x_j) - g(x_j) = \sigma (-1)^j \|f - g\|_C,
\end{equation*}
where $\sigma \in \{-1, 1\}$.
\end{theorem}

Let us consider the problem of uniform approximation
\begin{equation}
\label{eq:uniform_approximation}
    \|Vu - a\|_\infty \to \min\limits_{u \in \mathbb{R}^r},
\end{equation}
where $V \in \mathbb{R}^{n \times r}$ and $a \in \mathbb{R}^n$. In this case, we can prove the result similar to \thm{classical_equioscillation_theorem}. In \cite{zamarashkin2022best} the authors provide such a statement as the necessary condition, but here we show that this is precisely the criterion. To begin with, we need the following
\begin{lemma}
\label{lemma:residual_alternating_signs}
Let us condiser the problem \eq{uniform_approximation} with a Chebyshev matrix $V \in \mathbb{R}^{(r+1) \times r}$ and a vector $a \in \mathbb{R}^{r+1}$ that does not belong to the image of $V$. Let $\hat{u} \in \mathbb{R}^{r}$. Let us denote the residual by $w = a - V\hat{u}$. Then $\hat{u}$ is the solution of \eq{uniform_approximation} if and only if $|w_1| = |w_2| = \dots = |w_{r+1}| = \|w\|_\infty$
and the signs in the sequence $w_1 D_1(V), w_2 D_2(V), \dots, w_{r+1} D_{r+1}(V)$
alternate.

\begin{proof}
Let $\hat{u} \in \mathbb{R}^r$ be the solution of \eq{uniform_approximation}. By \thm{dzyadyk_sol}~\ref{dzyadyk_sol_ii} we have
\begin{equation*}
    \hat{u} = \sum\limits_{j=1}^{r+1} \dfrac{|D_j(V)|}{\sum\limits_{k=1}^{r+1} |D_k(V)|} \hat{u}^{(j)},
\end{equation*}
where $\hat{u}^{(j)}$ is the solution to the problem $V^{\setminus j} u = a_{\setminus j}$.
Moreover, by \thm{dzyadyk_sol}~\ref{dzyadyk_sol_i} we have 
\begin{equation}
\label{eq:dot_product_exp}
    (\hat{u}^{(j)})^T v^j - a_j = \dfrac{(-1)^{j+1}}{D_j(V)} x, \quad \text{ where } \quad x = \sum\limits_{k=1}^{r+1} (-1)^{k} a_{k} D_{k}(V)
\end{equation}
and $v^j$ is $j$--th row of the matrix $V$. Let us define $\tilde{a} \in \mathbb{R}^{n+1}$ such that $\tilde{a}_j = (\hat{u}^{(j)})^T v^j$.
By the definition of $\hat{u}^{(j)}$, we have
\begin{equation*}
    V \hat{u}^{(j)} = \begin{bmatrix}
        a_1 &
        a_2 &
        \dots &
        a_{j-1} &
        \tilde{a}_j &
        a_{j+1} &
        \dots &
        a_{n+1}
    \end{bmatrix}^T.
\end{equation*}
Let us denote
\begin{equation*}
    D = \diag{(D_1(V), \dots, D_{r+1}(V))}\text{ and }\hat{D} = \diag{(|D_1(V)|, \dots, |D_{r+1}(V)|)}.
\end{equation*}
Then
\begin{equation*}
    V\hat{u} = \sum\limits_{j=1}^{r+1} \dfrac{|D_j(V)|}{\sum\limits_{k=1}^{r+1} |D_k(V)|} (a + (\tilde{a}_j - a_j)e_j) = a + \dfrac{\hat{D} (\tilde{a} - a)}{\sum\limits_{k=1}^{r+1} |D_k(V)|}.
\end{equation*}
Then we have
\begin{equation*}
    w = a - V\hat{u} = \hat{D} (a - \tilde{a}) / \left(\sum\limits_{k=1}^{r+1} |D_k(V)|\right).
\end{equation*}    
Note that by \eq{dot_product_exp} we have $a - \tilde{a} = x S D^{-1} e$, where 
\begin{equation*}
    S = \diag{((-1)^1, (-1)^2, \dots, (-1)^{r+1})}
\end{equation*}
and $e = \begin{bmatrix}
    1 & 1 & \dots & 1
\end{bmatrix}^T$. Thus,
\begin{equation*}
    w = \dfrac{x}{\sum\limits_{k=1}^{r+1} |D_k(V)|} \hat{D} S D^{-1} e,
\end{equation*}
whence
\begin{equation*}
    w_j D_j = \dfrac{x}{\sum\limits_{k=1}^{r+1} |D_k(V)|} (-1)^j |D_j|
\end{equation*}
and the signs in the sequence $w_1 D_1, w_2 D_2, \dots, w_{r+1} D_{r+1}$
alternate. Moreover, by \lem{rp1points} all absolute values of the elements of $w$ are equal.

Let us prove the reversal. Let $\hat{u} \in \mathbb{R}^r$ be a vector, such that for $w = a - V\hat{u}$ we have $|w_1| = |w_2| = \dots = |w_{r+1}| = \|w\|_\infty$
and the signs in the sequence
\begin{equation*}
    w_1 D_1(V), w_2 D_2(V), \dots, w_{n+1} D_{r+1}(V)
\end{equation*}
alternate. Let us show that $\hat{u}$ is the solution of \eq{uniform_approximation}. From the assumptions on $w$, we have that
\begin{equation}
\label{eq:wj_values}
	w_j = c(-1)^j \sign D_j(V).
\end{equation}
By \thm{remez_opt_crit}, we have that $\hat{u}$ is the solution of \eq{uniform_approximation} if and only if there is nonzero $\delta \in \mathbb{R}^{r+1}$ with non-negative components such that $V^T \diag{(\sign w)} \delta = 0$.
Substituting \eq{wj_values} we have
\begin{equation}
\label{eq:tratata}
    V^T S D \hat{D}^{-1} \delta = 0.
\end{equation}
Let us consider the matrix $\tilde{V}^{(k)} \in \mathbb{R}^{(r+1)\times(r+1)}$, which corresponds to the matrix $V$ with copied $k$--th column,
\begin{equation*}
	\tilde{V}^{(k)} = \begin{bmatrix}
	v_1 & v_2 & \dots & v_{k-1} & v_k & v_k & v_{k+1} & \dots & v_{r+1}
	\end{bmatrix}.
\end{equation*}
Clearly, $\tilde{V}_k$ is singular. Let us apply Laplace expansion to the $k$--th column of $\tilde{V}_k$. Then
\begin{equation}
\label{eq:laplace_exp}
	\sum\limits_{j=1}^{r+1} (-1)^{k+j} D_j(V) v_{jk} = 0.
\end{equation}
Then choosing $\delta_j = |D_j|$ we get in the left hand side of \eq{tratata}
\begin{equation*}
    V^T S D \hat{D}^{-1} \delta = V^T S D e,
\end{equation*}
which is equal to zero by \eq{laplace_exp}.
\end{proof}
\end{lemma}

Finally, we are ready to prove the equioscillation theorem for the problem \eq{uniform_approximation}.
\begin{theorem}
\label{theorem:residual_alternating_signs}
Let us consider the problem \eq{uniform_approximation} with a Chebyshev matrix $V \in \mathbb{R}^{n \times r}$ and a vector $a \in \mathbb{R}^{n}$ that does not belong to the image of $V$. Let $\hat{u} \in \mathbb{R}^{r}$. Let us denote the residual by $w = a - V\hat{u}$. Then $\hat{u}$ is the solution of \eq{uniform_approximation} if and only if there is a set of integers $1 \le j_1 < j_2 < \dots < j_{r+1} \le n$ such that
\begin{equation*}
	|w_{j_1}| = |w_{j_2}| = \dots = |w_{j_{r+1}}| = \|w\|_\infty
\end{equation*}
and the signs in the sequence
\begin{equation*}
    w_{j_1} \Delta_1, w_{j_2} \Delta_2, \dots, w_{j_{r+1}} \Delta_{r+1}
\end{equation*}
alternate, where
$\Delta_k = \det V((j_1, \dots, j_{k-1}, j_{k+1}, \dots, j_{r+1}))$.
\begin{proof}
Let $\hat{u}$ be the solution of \eq{uniform_approximation}. Then by \thm{charact_set_size} there is a characteristic set of $r+1$ elements $j_1 < j_2 < \dots < j_{r+1}$. Let us denote $J' = (j_1, j_2, \dots, j_{r+1})$. By definition of characteristic set, $\hat{u}$ is the solution to the problem
\begin{equation*}
\|a(J') - V(J') u\|_\infty \to \min\limits_{u \in \mathbb{C}^{r}},
\end{equation*}
therefore, the conditions of \lem{residual_alternating_signs} are met for $\hat{u}$, which implies the statement of the theorem.

Let a set $J' = (j_1, j_2, \dots, j_{r+1})$ satisfy the conditions of the theorem. Then by \lem{residual_alternating_signs} the vector $\hat{u}$ is the solution to the problem
\begin{equation*}
\|a(J') - V(J') u\|_\infty \to \min\limits_{u \in \mathbb{R}^{r}},
\end{equation*}
whence by \thm{remez_opt_crit} there is a non-zero $\delta \in \mathbb{R}^{r+1}$ with non-negative components such that
\begin{equation*}
	V(J')^T \diag{(\sign w(J'))} \delta = 0.
\end{equation*}
Let
\begin{equation*}
	\hat{J} = \{j \in \{1, 2, \dots, n\} : |w_j | = \|w\|_\infty \}.
\end{equation*}
Clearly, $J' \subset \hat{J}$. Let us denote by $E_{(k)}$ and $J'_{(k)}$ $k$--th smallest element of sets $\hat{J}$ and $J'$ respectively. Let us set $\hat{\delta} \in \mathbb{R}^{|\hat{J}|}$ such that
\begin{equation*}
    \hat{\delta}_j = \begin{cases}
        \delta_k, & \hat{J}_{(j)} \in J' \text{ and } J'_{(k)} = \hat{J}_{(j)} \\
        0, & \text{otherwise}.
    \end{cases}
\end{equation*}
Then
\begin{equation*}
	V(\hat{J})^T \diag{(\sign w(\hat{J}))} \hat{\delta} = V(J')^T \diag{(\sign w(J'))} \delta = 0,
\end{equation*}
whence by \thm{remez_opt_crit} $\hat{u}$ is the solution to the problem \eq{uniform_approximation}.
\end{proof}
\end{theorem}
There is a direct connection between \thm{classical_equioscillation_theorem} and \thm{residual_alternating_signs}. Let us consider the system of points $x_0, \dots, x_{d+1}$ on a segment $[a, b] \subset \mathbb{R}$, namely,
\begin{equation*}
    a \le x_0 < x_1 < \dots < x_{d + 1} \le b
\end{equation*}
and the Vandermonde matrix $W(x_0, \dots, x_{d+1}) \in \mathbb{R}^{(d+2) \times (d+1)}$ constructed on this points. It is known that
\begin{equation*}
	\det W(y_1, \dots, y_k) = \prod\limits_{i < j} (y_i - y_j)
\end{equation*}
for a square Vandermonde matrix. Therefore, the determinants for all submatrices $W(x_0, \dots, x_{d+1})^{\setminus j}$ have the same sign, so \thm{residual_alternating_signs} reduces to \thm{classical_equioscillation_theorem}.

\section{Accelerated uniform approximation algorithm}
\label{sec:accelerated_alg}
Now we are ready to discuss an accelerated algorithm for solving the problem of the best uniform approximation
\begin{equation}
\label{eq:accelerated_basic_problem}
    \|Vu - a\|_\infty \to \min\limits_{u \in \mathbb{R}^r}.
\end{equation}
for a Chebyshev matrix $V \in \mathbb{R}^{n \times r}$ and a vector $a \in \mathbb{R}^n$. To the best of our knowledge, the first algorithm for solving \eq{accelerated_basic_problem} was suggested in \cite{zamarashkin2022best}, where the authors propose an iterative procedure for constructing the characteristic set (or, equivalently, to the set that provides equioscillation properties as in \thm{residual_alternating_signs}). We remind that the size of the characteristic set is always $r+1$ for the problem with the Chebyshev matrix, so when the characteristic set is found, we only need to solve the problem of size $(r+1)\times r$, which can be done by explicit formulas (see \thm{dzyadyk_sol}). The complexity of the method proposed in \cite{zamarashkin2022best} is $O(I (nr + r^4))$, where $I$ is the number of iterations. In this paper, we propose the algorithm with the complexity $O(r^3 + Inr)$, which is equivalent to one described in \cite{zamarashkin2022best} in precise arithmetic. The core idea is to support the QR decomposition for the current submatrix of size $(r+1)\times r$. This idea draws the inspiration from the similar approach for maximum volume algorithm \cite{osinsky2018rectangular}.

\subsection{On the problem of size $(r+1) \times r$}
\label{sec:small_problem_optimized}

If the QR decomposition of a matrix is known, then the solution to the least squares problem with this matrix can be found in $O(r^2)$ operation, where $r$ is the size of the matrix. Similarly, solution to the best uniform approximation problem can be found in $O(r^2)$ for a matrix of size $(r+1) \times r$ if the QR decomposition of the matrix is known. Let $\hat{V} \in \mathbb{R}^{(r+1)\times r}$ be a Chebyshev matrix and $\hat{a} \in \mathbb{R}^{r+1}$ be the right-hand side for the best uniform approximation problem. Let $\hat{V} = \hat{Q} \hat{R}$, where $\hat{Q} \in \mathbb{R}^{(r+1) \times r}$ has orthonormal columns and $\hat{R}$ is non-singular upper triangular matrix. Let also $\hat{q}'$ denote the vector such that $\begin{bmatrix}
    \hat{Q} & \hat{q}'
\end{bmatrix} \in \mathbb{R}^{(r+1)\times (r+1)}$ is orthogonal.

To solve the problem \eq{accelerated_basic_problem}, we can use \thm{remez_opt_crit}, which provides the criterion for a vector to be the solution to the problem. By \lem{rp1points} and \thm{remez_opt_crit} a vector $\hat{u} \in \mathbb{R}^r$ is the solution to the best uniform approximation problem if and only if there is a non-zero vector $\delta \in \mathbb{R}^{r+1}$ with non-negative components such that
\begin{equation}
\label{eq:remez_result}
    \hat{V}^T \diag{(\sign{\hat{w}})} \delta = 0,
\end{equation}
where $\hat{w} = \hat{a} - \hat{V}\hat{u}$. Since $\hat{V}$ is Chebyshev, the dimension of $\ker \hat{V}^T$ is equal to 1 and the generating vector of the space $\ker \hat{V}^T$ is $\hat{q}'$ (note that $\hat{q}'$ has non-zero components by \lem{orthog_complement}). Then \eq{remez_result} can be satisfied only if $\delta_k = |\hat{q}'_k|$ and $\sign \hat{w}_k = c \sign {\hat{q}'_k}$ for $k=1,2, \dots, r+1$, where $c=\pm 1$ and does not depend on $k$. If $\hat{u}$ is the solution, then by \lem{rp1points} we have $|\hat{w}_k| = \|\hat{w}\|_\infty$ for $k = 1,2,\dots,r+1$. Therefore, $\hat{w}_j = \hat{c} \sign \hat{q}'_j$, where $\hat{w}$ is the residual for the optimal solution. Since $\hat{w}$ is the residual, the equation $\hat{V}\hat{u} = \hat{a} - \hat{w}$ should have a solution.
\begin{equation}
\label{eq:solution_for_residual_must_exists}
    \hat{a} - \hat{w} = \hat{Q}\hat{R}\hat{u} = \begin{bmatrix}
        \hat{Q} & \hat{q}'
    \end{bmatrix} \begin{bmatrix}
        \hat{R}\\
        0
    \end{bmatrix}\hat{u}.
\end{equation}
Then for \eq{solution_for_residual_must_exists} to have a solution we need $(\hat{a} - \hat{w})^T \hat{q}' = 0$, hence $\hat{c} (\sign \hat{q'})^T \hat{q}' = \hat{a}^T\hat{q}'$ and $\hat{c} = \dfrac{\hat{a}^T\hat{q}'}{\|\hat{q}'\|_1}$. Therefore, we can compute the residual for the optimal solution. Once the residual is computed, we can find the solution as $\hat{u} = \hat{R}^{-1}\hat{Q}^T(\hat{a} - \hat{w})$.

The final procedure is presented in \alg{small_problem_optimized}. Clearly, the complexity of the described method is $O(r^2)$.

\begin{algorithm}
\caption{Accelerated algorithm for the problem of size $(r + 1) \times r$.}
\label{alg:small_problem_optimized}
\begin{algorithmic}
\Require{Matrix with orthonormal columns $\hat{Q} \in \mathbb{R}^{(r+1) \times r}$, complement to the orthogonal matrix $\hat{q}' \in \mathbb{R}^{r+1}$, upper triangular matrix $\hat{R} \in \mathbb{R}^{r \times r}$, right-hand side vector $\hat{a} \in \mathbb{R}^{r+1}$.}
\Ensure{$\hat{u} \in \mathbb{R}^r$ solution to the problem $\|\hat{Q}\hat{R} u - a\|_\infty \to \min\limits_{u \in \mathbb{R}^r}$.}

\State{$\hat{c} = \hat{a}^T\hat{q}' / \|\hat{q}'\|_1$}
\State{$\hat{w} = \hat{c} \sign{\hat{q}'}$}
\State{$\hat{u} = \texttt{solve\_triangular}(\hat{R}, \hat{Q}^T (\hat{a} - \hat{w}))$}
\end{algorithmic}
\end{algorithm}


\subsection{Iterative update of the current set of indices}
\label{sec:iterative_update}

Let $V \in \mathbb{R}^{n \times r}$ be a Chebyshev matrix and $a \in \mathbb{R}^n$ be a vector. Let $\hat{J}$ be a set of row indices of the matrix $V$ and $|\hat{J}| = r + 1$. Let us denote $\hat{V} = V(\hat{J})$ and $\hat{a} = a(\hat{J})$. Using \alg{small_problem_optimized} we can find the solution to the problem with the matrix $\hat{V}$ and right-hand side $\hat{a}$. Let us denote the solution by $\hat{u}$ and $w = a - V\hat{u}$. From \thm{remez_opt_crit} it can be shown that if $\|w(\hat{J})\|_\infty = \|w\|_\infty$, then $\hat{u}$ is the solution to the best uniform approximation problem with the matrix $V$ and right-hand side $a$. Otherwise, let us denote by $\tilde{j}$ the position of the maximum absolute value element in $w$. It can be shown (see \cite[Theorem 12]{zamarashkin2022best}) that there is a set $\tilde{J}$, obtained from $\hat{J}$ by the replacement of one element with $\tilde{j}$ such that
\begin{equation*}
    \min\limits_{u \in \mathbb{R}^r}\| V(\hat{J})u - a(\hat{J}) \|_\infty < \min\limits_{u \in \mathbb{R}^r}\| V(\tilde{J})u - a(\tilde{J}) \|_\infty.
\end{equation*}
Since the value of the approximation error cannot increase indefinitely, in a finite number of steps the iterative procedure converges to the solution of the best uniform approximation problem.

In this section, we propose an efficient procedure for the update of the set of indices $\hat{J}$.
We need to select some row of $\hat{V}$ to replace. The choice depends on the value of the minimum \eq{min_value} in \lem{orthog_solution_formula}.
Let the matrix $\tilde{V} \in \mathbb{R}^{(r+1)\times r}$ be obtained from the matrix $\hat{V}$ by replacing the $k$--th row with a vector $h$ and the vector $\tilde{a} \in \mathbb{R}^{r+1}$ be obtained from the vector $\hat{a}$ by replacing the $k$--th element with $\xi \in \mathbb{R}$.
The computation by \eq{min_value} requires the knowledge of a vector $\tilde{q}'$, such that $\tilde{V}^T \tilde{q}' = 0$. Thus, we need to be able to quickly compute all these $r+1$ vectors for $k$ from $1$ to $r+1$, before we update the QR decomposition.
By definition $\tilde{V} = \hat{V} + e_k(h - \hat{V}^Te_k)^T$. Let us construct a non-zero vector $\tilde{q}' \in \mathbb{R}^{r+1}$ such that $\tilde{V}^T\tilde{q}' = 0$. Let $\tilde{V} = \tilde{Q} \tilde{R}$, where $\tilde{Q} \in \mathbb{R}^{(r+1) \times r}$ has orthonormal columns and $\tilde{R}$ is a non-singular upper triangular matrix. Then
\begin{equation*}
    \hat{Q} \hat{R} + e_k (h - \hat{V}^T e_k)^T = \tilde{Q} \tilde{R}.
\end{equation*}
Let us multiply the last equation by $\begin{bmatrix}
    \hat{Q} & \hat{q}'
\end{bmatrix}^T$ on the left and by $\hat{R}^{-1}$ on the right.
\begin{equation*}
    \begin{bmatrix}
        I\\
        0
    \end{bmatrix} + \begin{bmatrix}
    \hat{Q} & \hat{q}'
\end{bmatrix}^T e_k (h - \hat{V}^T e_k)^T \hat{R}^{-1} = \begin{bmatrix}
    \hat{Q} & \hat{q}'
\end{bmatrix}^T\tilde{Q} \tilde{R}\hat{R}^{-1}.
\end{equation*}
Let us denote $z = \hat{R}^{-T}(h - \hat{V}^Te_k) = \hat{R}^{-T}h - \hat{q}^k$, where $\hat{q}^k \in \mathbb{R}^r$ denotes the $k$--th row of the matrix $\hat{Q}$. Then
\begin{equation}
\label{eq:some_eq1}
    \begin{bmatrix}
        I\\
        0
    \end{bmatrix} + \begin{bmatrix}
    \hat{q}^k \\
    \hat{q}'_k
\end{bmatrix} z^T = \begin{bmatrix}
    \hat{Q} & \hat{q}'
\end{bmatrix}^T\tilde{Q} \tilde{R}\hat{R}^{-1}.
\end{equation}
Note that $\tilde{V}^T\tilde{q}' = 0$ is equivalent to
\begin{equation*}
    (\tilde{q}')^T \begin{bmatrix}
    \hat{Q} & \hat{q}'
\end{bmatrix} \begin{bmatrix}
    \hat{Q} & \hat{q}'
\end{bmatrix}^T \tilde{Q} \tilde{R} \hat{R}^{-1} = 0.
\end{equation*}
Substituting \eq{some_eq1} into the last equation we obtain
\begin{equation}
\label{eq:q_tilde_prime_formula}
    (\tilde{q}')^T \begin{bmatrix}
    \hat{Q} & \hat{q}'
\end{bmatrix} \left( \begin{bmatrix}
        I\\
        0
    \end{bmatrix} + \begin{bmatrix}
    \hat{q}^k \\
    \hat{q}'_k
\end{bmatrix} z^T \right) = 0.
\end{equation}
Let us denote $x = \hat{Q}^T \tilde{q}'$ and $\alpha = (\hat{q}')^T \tilde{q}'$. Then
\begin{equation}
\label{eq:some_eq2}
    \begin{bmatrix}
        x^T & \alpha
    \end{bmatrix} \left( \begin{bmatrix}
        I\\
        0
    \end{bmatrix} + \begin{bmatrix}
    \hat{q}^k \\
    \hat{q}'_k
\end{bmatrix} z^T \right) = 0,
\end{equation}
which is equivalent to
\begin{equation*}
    x + (x^T \hat{q}^k + \alpha \hat{q}'_k)z = 0,
\end{equation*}
whence $x = cz$, where $c \in \mathbb{R}$. Then $cz + (cz^T \hat{q}^k + \alpha \hat{q}'_k)z = 0$ and if $z \neq 0$, we have
\begin{equation*}
    c + cz^T \hat{q}^k + \alpha \hat{q}'_k = 0.
\end{equation*}
Therefore, $\alpha \hat{q}'_k = -c(1 + z^T \hat{q}^k)$. Then
\begin{equation*}
    \begin{bmatrix}
        \hat{q}'_k x \\
        \hat{q}'_k \alpha
    \end{bmatrix} = \begin{bmatrix}
        \hat{q}'_k c z \\
        -c(1 + z^T \hat{q}^k)
    \end{bmatrix} = c \begin{bmatrix}
        -\hat{q}'_k z \\
        1 + z^T \hat{q}^k
    \end{bmatrix}.
\end{equation*}
Since we look for \textit{any} $\tilde{q}' \neq 0$ such that $\tilde{V}^T\tilde{q}' = 0$, we do not care about the norm of $\tilde{q}'$. Then we can take
\begin{equation}
\label{eq:q_tilde_prime_expr}
    \tilde{q}' = \begin{bmatrix}
        \hat{Q} & \hat{q}'
    \end{bmatrix} \begin{bmatrix}
        -\hat{q}'_k z \\
        1 + z^T \hat{q}^k
    \end{bmatrix}.
\end{equation}
Note that $\hat{q}'_k \neq 0$ by \lem{orthog_complement}, so $\tilde{q}'$ is always non-zero. Substituting \eq{q_tilde_prime_expr} to \eq{q_tilde_prime_formula} gives identically zero, so \eq{q_tilde_prime_expr} is correct for any values of $z$. Once we have computed $\tilde{q}'$, due to \lem{orthog_solution_formula}, we can calculate
\begin{equation*}
    \min\limits_{u \in \mathbb{R}^r}\| V(\tilde{J})u - a(\tilde{J}) \|_\infty = \dfrac{|\tilde{a}^T \tilde{q}'|}{\|\tilde{q}'\|_1}.
\end{equation*}

Direct computation of \eq{q_tilde_prime_expr} requires $O(r^2)$ operation, however, the complexity can be reduced further if we perform some operations, common for all values of $k$, in advance. Let us denote $g = \hat{R}^{-T}h$ and $y = \hat{Q}g$. Both vectors can be precomputed in $O(r^2)$ operation and do not depend on the row number $k$. Then $z = g - \hat{q}^k$ and
\begin{equation*}
    \tilde{q}' = \begin{bmatrix}
        \hat{Q} & \hat{q}'
    \end{bmatrix} \begin{bmatrix}
        -\hat{q}'_k (g - \hat{q}^k) \\
        1 + (g - \hat{q}^k)^T \hat{q}^k
    \end{bmatrix} = -\hat{q}'_k y + \hat{q}'_k \hat{Q}\hat{q}^k + (1 + (g - \hat{q}^k)^T \hat{q}^k) \hat{q}'.
\end{equation*}
Note that $\hat{Q} \hat{q}^k = e_k - \hat{q}'_k \hat{q}'$, whence
\begin{equation*}
    \tilde{q}' = -\hat{q}'_k y + \hat{q}'_k (e_k - \hat{q}'_k \hat{q}') + \hat{q}' + (g^T\hat{q}^k) \hat{q}' - ((\hat{q}^k)^T \hat{q}^k) \hat{q}'.
\end{equation*}
Since $(\hat{q}^k)^T \hat{q}^k = 1 - (\hat{q}'_k)^2$ and $g^T \hat{q}^k = y_k$, we obtain
\begin{equation*}
    \tilde{q}' = \hat{q}'_k (e_k - y) + y_k \hat{q}'.
\end{equation*}
Thus, the vector $\tilde{q}'$ can be computed in $O(r)$ operations if the vector $y$ is known. Then we can compute the residual norm when replacing $k$--th element in the set $\hat{J}$ in $O(r)$ operations, hence we can select replacement which maximizes $\min\limits_{u \in \mathbb{R}^r}\| V(\tilde{J})u - a(\tilde{J}) \|_\infty$ in $O(r^2)$ operations. The final procedure of finding the best replacement is presented in \alg{replacement_optimized}.

\begin{algorithm}
\caption{Optimal replacement of the row in the matrix of size $(r + 1) \times r$.}
\label{alg:replacement_optimized}
\begin{algorithmic}
\Require{Matrix with orthonormal columns $\hat{Q} \in \mathbb{R}^{(r+1) \times r}$, complement to the orthogonal matrix $\hat{q}' \in \mathbb{R}^{r+1}$, upper triangular matrix $\hat{R} \in \mathbb{R}^{r \times r}$, such that $\hat{V} = \hat{Q} \hat{R}$. Right-hand side vector  $\hat{a} \in \mathbb{R}^{r + 1}$, new row $h \in \mathbb{R}^r$, new right-hand side element $\xi \in \mathbb{R}$.}
\Ensure{$\hat{k}$ the row number in the matrix $\hat{V}$ that should be replaced by $h$ and $\hat{\mu}$ the new error value.}

\State{$y = \hat{Q} \cdot \texttt{solve\_triangular}(\hat{R}^T, h)$}
\State{$\hat{\mu} = -1$, $\hat{k} = -1$}

\For{$k=1,2,\dots,r+1$}
    \State{$\tilde{a} = \hat{a}$}
    \State{$\tilde{a}_k = \xi$}
    \State{$\tilde{q}' = \hat{q}_k'(e_k - y) + y_k \hat{q}'$}
    \State{$\mu = \dfrac{|(\tilde{q}')^T \tilde{a}|}{\|\tilde{q}'\|_1}$}
    \If{$\mu > \hat{\mu}$}
        \State{$ \hat{\mu} = \mu $}
        \State{$\hat{k} = k$}
    \EndIf
\EndFor
\end{algorithmic}
\end{algorithm}

After the element that needs to be replaced is found, we need to update the set indices and the QR decomposition for the corresponding submatrix. Let $\hat{V} = \hat{Q}\hat{R}$ and the matrix $\begin{bmatrix}
    \hat{Q} & \hat{q}'
\end{bmatrix}$ is orthogonal. Let the matrix $\tilde{V}$ be obtained from the matrix $\hat{V}$ by replacement of the $\hat{k}$--th row with $h \in \mathbb{R}^r$. We need to build the representation $\tilde{V} = \tilde{Q}\tilde{R}$ and a vector $\tilde{q}'$ such that $\begin{bmatrix}
    \tilde{Q} & \tilde{q'}
\end{bmatrix}$ is orthogonal and $\tilde{R}$ is non-singular upper triangular. Note that $\tilde{V} = \hat{V} + e_{\hat{k}}(h - \hat{v}^{\hat{k}})^T$ and $\hat{V} = \begin{bmatrix}
    \hat{Q} & \hat{q}'
\end{bmatrix}\begin{bmatrix}
    \hat{R}\\
    0
\end{bmatrix}$. Then the problem reduces to the classical rank-1 QR update, which can be done in $O(r^2)$ operations (see \cite[Section~6.5.1]{golub2013matrix}).

\subsection{Final algorithm}

In this section, we describe the algorithm, that solves the best uniform approximation problem with a Chebyshev matrix $V \in \mathbb{R}^{n\times r}$ and a vector $a \in \mathbb{R}^n$ in $O(r^3 + Inr)$ operations. Let us choose an arbitrary ordered set $\hat{J}_1$ of $r+1$ distinct row indices of the matrix $V$. Let $\hat{V} = V(\hat{J}_1)$ and $\hat{a} = a(\hat{J}_1)$. Let us build the complete QR decomposition $\hat{V} = \begin{bmatrix}
    \hat{Q} & \hat{q}'
\end{bmatrix}\begin{bmatrix}
    \hat{R} \\
    0
\end{bmatrix}.$ It can be done in $O(r^3)$ operations.

Let us compute the solution $\hat{u} \in \mathbb{R}^r$ to the problem $\| V(\hat{J}_1)u - a(\hat{J}_1) \|_\infty \to \min\limits_{u \in \mathbb{R}^r}$ and calculate the residual $w = a - V\hat{u}$. This step requires $O(nr)$ operations. If $\|w(\hat{J}_1)\|_\infty = \|w\|_\infty$, then $\hat{u}$ is the solution to the best uniform approximation problem with the matrix $V$ and the right-hand side $a$ by \thm{remez_opt_crit}. Otherwise, we have $\|w(\hat{J}_1)\|_\infty < \|w\|_\infty$ and we can find the position $\hat{j}$ of the maximum absolute value element in the vector $w$. Note that $\hat{j} \notin \hat{J}_1$. Using \alg{replacement_optimized} we can find the index $\hat{k}$ of element in the set $\hat{J}_1$, that needs to be replaced by $\hat{j}$ in $O(r^2)$ operations. Let us denote by $\hat{J}_2$ the set obtained from  $\hat{J}_1$ by the replacement of $\hat{k}$--th element with $\hat{j}$. In this case
\begin{equation*}
    \min\limits_{u \in \mathbb{R}^r}\| V(\hat{J}_1)u - a(\hat{J}_1) \|_\infty < \min\limits_{u \in \mathbb{R}^r}\| V(\hat{J}_2)u - a(\hat{J}_2) \|_\infty.
\end{equation*}
It remains to build the QR decomposition for the matrix $V(\hat{J}_2)$, which can be constructed from the QR decomposition for the matrix $V(\hat{J}_1)$ by the rank-1 QR update as described in \sect{iterative_update}. Then we can repeat the procedure for the set $\hat{J}_2$ and obtain $\hat{J}_3, \hat{J}_4,$ etc. Note that since $\hat{\mu}_t = \min\limits_{u \in \mathbb{R}^r}\| V(\hat{J}_t)u - a(\hat{J}_t) \|_\infty$ cannot increase indefinitely, in a finite number of steps the iterative procedure converges to the solution of the best uniform approximation problem. The final method is presented in \alg{final_accelerated_alg}. It is clear, that the complexity of the constructed algorithm is $O(r^3 + Inr)$.

\begin{algorithm}
\caption{Best uniform approximation algorithm.}
\label{alg:final_accelerated_alg}
\begin{algorithmic}
\Require{Chebyshev matrix $V \in \mathbb{R}^{n\times r}$, right-hand side $a \in \mathbb{R}^n$, initial ordered set $\hat{J}$.}
\Ensure{Solution $\hat{u} \in \mathbb{R}^r$ to the problem $\|Vu - a\|_\infty \to \min\limits_{u \in \mathbb{R}^r}$, characteristic set $\hat{J}$ of the problem.}

\State{$\hat{V} = V(\hat{J})$, $\hat{a} = a(\hat{J})$, $t=1$}
\State{$\hat{Q}, \hat{q}', \hat{R} \leftarrow \texttt{qr\_decomposition}(\hat{V})$} \Comment{$\hat{V} = \begin{bmatrix}
    \hat{Q} & \hat{q}'
\end{bmatrix}\begin{bmatrix}
    \hat{R} \\
    0
\end{bmatrix}$, $\begin{bmatrix}
    \hat{Q} & \hat{q}'
\end{bmatrix}$ is orthogonal}
\State{$\hat{u} \leftarrow \texttt{uniform\_approximation}(\hat{Q}, \hat{q}', \hat{R}, \hat{a})$  \Comment{\alg{small_problem_optimized}}}
\State{$w = a - V\hat{u}$}
\While{$\|w(\hat{J})\|_\infty < \|w\|_\infty$}
    \State{$\hat{j} \leftarrow \argmax\limits_{j \in \{1, \dots, n\}} |w_j|$}
    \State{$\hat{k} \leftarrow \texttt{best\_replacement}(\hat{Q}, \hat{q}', \hat{R}, v^{\hat{j}}, a_{\hat{j}})$ \Comment{\alg{replacement_optimized}}}
    \State{$\hat{a}_{\hat{k}} = a_{\hat{j}}$}
    \State{Replace $\hat{k}$--th row of the matrix $\hat{V}$ with $v^{\hat{j}}$}
    \State{Update QR decomposition factors $\hat{Q}$, $\hat{q}'$ and $\hat{R}$ for the matrix $\hat{V}$ \Comment{Rank-1 QR update}}
    \State{Replace $\hat{k}$--th element of the ordered set $\hat{J}$ with $\hat{j}$}
    \State{$\hat{u} \leftarrow \texttt{uniform\_approximation}(\hat{Q}, \hat{q}', \hat{R}, \hat{a})$  \Comment{\alg{small_problem_optimized}}}
    \State{$w = a - V\hat{u}$}
\EndWhile
\end{algorithmic}
\end{algorithm}

\section{Alternating minimization method}
\label{sec:alternance}
In this section we describe the alternating minimization method for building low-rank Chebyshev approximations and reveal the properties of its limit points.

\subsection{Method description}
\label{sec:alternating_minim_description}

Let $A \in \mathbb{R}^{m \times n}$ be a matrix. Here and further we assume that the sizes $m$ and $n$ are strictly greater than $1$. Our goal is to construct a low-rank entrywise approximation with rank $r$, namely,
\begin{equation}
\label{eq:main_cheb_matrix_probem}
    \|A - U V^T\|_C \to \min\limits_{U \in \mathbb{R}^{m \times r}, V \in \mathbb{R}^{n \times r}}.
\end{equation}
We also assume that $\rank A > r$. The problem \eq{main_cheb_matrix_probem} is difficult to solve directly \cite{gillis2019low}, so to tackle it we assume that one of the matrices ($U$ or $V$) is known. Then we can consider the problem
\begin{equation}
\label{eq:one_matrix_cheb_problem}
    \|A - U V^T\|_C \to \min\limits_{U \in \mathbb{R}^{m \times r}},
\end{equation}
which can be decomposed to the set of problems of the form
\begin{equation*}
    \|a - Vu\|_\infty \to \min\limits_{u \in \mathbb{R}^r},
\end{equation*}
where $V \in \mathbb{R}^{n \times r}$ and $a \in \mathbb{R}^n$. Let $V \in \mathbb{R}^{n \times r}$ be a Chebyshev matrix. Then there is a unique map (see \thm{exists_unique_cont} for the correctness) $\phi: \mathbb{R}^{m\times n} \times \mathbb{R}^{n \times r} \to \mathbb{R}^{m \times r}$ such that $\phi(A, V)^i = \argmin\limits_{x \in \mathbb{R}^r} \| a^i - Vx \|_\infty$,
where with superscript we denote the row of a matrix and $a^i \in \mathbb{R}^n$ is the $i$--th row of the matrix $A$. Note that $\phi(A, V)$ is a solution of \eq{one_matrix_cheb_problem} (however, the solution of \eq{one_matrix_cheb_problem} may be not unique). Similarly, we define the map
$\psi: \mathbb{R}^{m\times n} \times \mathbb{R}^{m \times r} \to \mathbb{R}^{n \times r}$ such that $\psi(A, U)^j = \argmin\limits_{x \in \mathbb{R}^r} \| a_j - Ux \|_\infty$,
where $a_j$ is the $j$--th column of the matrix $A$. $\psi(A, U)$ is a solution to the problem
\begin{equation*}
    \|A - U V^T\|_C \to \min\limits_{V \in \mathbb{R}^{n \times r}}.
\end{equation*}

\begin{definition}
\label{definition:alternating_minimization}
    Let $A \in \mathbb R^{m \times n}$ be a matrix. We say that the pair of sequences of Chebyshev matrices $\{U^{(t)} \in \mathbb{R}^{m \times r}\}_{t \in \mathbb N}$ and $\{V^{(t)} \in \mathbb{R}^{n \times r}\}_{t \in \mathbb N}$ is obtained by the \textit{alternating minimization method} for the matrix $A$ with the initial point $V^{(0)}$, where $V^{(0)} \in \mathbb{R}^{n \times r}$ is a Chebyshev matrix, if
    \begin{equation*}
    \begin{cases}
     U^{(t)} = \phi(A, V^{(t-1)}), \\
     V^{(t)} = \psi(A, U^{(t)}) \\
    \end{cases}
    \end{equation*}
    for all $t \in \mathbb N$.
\end{definition}

Note that if a matrix $V$ is Chebyshev, it does not imply that $\phi(A, V)$ is also Chebyshev. In \cite{morozov2023optimal} the authors show that for rank-1 approximation for almost all matrices $A$ if $V$ is Chebyshev, then $\phi(A, V)$ is also Chebyshev. However, it is no longer true for arbitrary rank approximation. Moreover, we conjecture that for $r \ge 2$ for almost all matrices $A \in \mathbb{R}^{m \times n}$ there exists a matrix $V \in \mathbb{R}^{n \times r}$ such that $\phi(A, V)$ is not Chebyshev. Fortunately, our numerical experiments demonstrate that such situations are rare. Nevertheless, when we apply \dfn{alternating_minimization} in our theoretical derivations, we need to explicitly assume that the generated matrices are Chebyshev.

Let us formulate the basic properties of the alternating minimization method.

\begin{lemma}
\label{lemma:basic_prop}
Let $A \in \mathbb{R}^{m \times n}$ and matrix $V^{(0)} \in \mathbb{R}^{n \times r}$ be Chebyshev. Let the pair of sequences $\{U^{(t)} \in \mathbb{R}^{m\times r}\}_{t \in \mathbb N}$ and $\{V^{(t)} \in \mathbb{R}^{n\times r}\}_{t \in \mathbb N}$ be generated by alternating minimization method for the matrix $A$ and the initial point $V^{(0)}$. Then the following statements hold.
\begin{enumerate}[label=(\roman*)]
    \item\label{basic_prop_i}
    We have
    \begin{equation*}
        \|A - U^{(t)} (V^{(t-1)})^T\|_C \ge \|A - U^{(t)} (V^{(t)})^T\|_C \ge \|A - U^{(t+1)} (V^{(t)})^T\|_C
    \end{equation*}
    for all $t \in \mathbb N$.
    \item\label{basic_prop_iii} If the pair of sequences $\{\tilde{U}^{(t)}\}_{t \in \mathbb N}$ and $\{\tilde{V}^{(t)}\}_{t \in \mathbb N}$ is obtained by the alternating minimization method for the matrix $A$ and the initial point $\alpha V^{(0)}$, where $\alpha \ne 0$, then $\tilde{U}^{(t)} = 1/\alpha\; U^{(t)}$ and $\tilde{V}^{(t)} = \alpha V^{(t)}$.
\end{enumerate}
\begin{proof}
By the construction of $\phi$ we have
\begin{equation*}
    \inf\limits_{U \in \mathbb R^{m \times r}} \|A - U (V^{(t)})^T\|_C = \|A - \phi(A,V^{(t)}) (V^{(t)})^T\|_C,
\end{equation*}
whence we have
\begin{equation*}
    \|A - U^{(t)} (V^{(t)})^T\|_C \ge \|A - U^{(t+1)} (V^{(t)})^T\|_C
\end{equation*}
since $U^{(t+1)} = \phi(T,V^{(t)})$. The other inequality in the statement~\ref{basic_prop_i} can be proved similarly.

The statement~\ref{basic_prop_iii} follows from the uniqueness of the solution to the best uniform approximation problem (see \thm{exists_unique_cont}).
\end{proof}
\end{lemma}

Let  $V \in \mathbb R^{n \times r}$ be a Chebyshev matrix and a pair of sequences $\{U^{(t)}\}_{t \in \mathbb N}$ and $\{V^{(t)}\}_{t \in \mathbb N}$ be obtained by the alternating minimization method for a matrix $A$ and an initial point $V^{(0)} = V$. From \lem{basic_prop}~\ref{basic_prop_i} it follows that the sequence $\|A - U^{(t)} (V^{(t)})^T\|_C$ is non-increasing and, since it consists only of non-negative numbers, converges. We shall denote its limit by $E(A, V)$. The following lemma contains the elementary properties of this function.

\begin{lemma}
\label{lemma:basic_prop_E}
Let $A \in \mathbb R^{m \times n}$ be a matrix and a matrix $V \in \mathbb R^{n\times r}$ be Chebyshev. Let also the alternating minimization method for the matrix $A$ and the initial point $V^{(0)} = V$ be correct (that is the generated matrices are Chebyshev). Then the following statements hold.
\begin{enumerate}[label=(\roman*)]
    \item\label{basic_prop_E_i} $E(A,V) \ge 0$ and $E(A,V) = E(A,\alpha V)$ for $\alpha \ne 0$.
    \item\label{basic_prop_E_ii} $E(A,V) = E(A,\tilde{V})$, where $\tilde{V} = \psi(A, \phi(A,V))$.
    \item\label{basic_prop_E_iii} The function $E(A,V)$ is upper semi-continuous with respect to $V$.
\end{enumerate}
\begin{proof}
    The statements~\ref{basic_prop_E_i} and~\ref{basic_prop_E_ii} directly follow from the definition of $E(T,V)$ and \lem{basic_prop}. The upper semi-continuity holds, because $E(A, V)$ is a limit of a decreasing sequence of continuous functions ($\phi$ and $\psi$ are continuous by \thm{exists_unique_cont}).
\end{proof}
\end{lemma}

The final procedure of the alternating minimization method is presented in the \alg{alternating_minimization}. Note that we also apply re-normalizations after each iteration of the method since, according to \lem{basic_prop_E}, they do not influence the solution, but improve the numerical stability.

Note that the question of the convergence of sequences $\{U^{(t)}\}_{t \in \mathbb N}$ and $\{V^{(t)}\}_{t \in \mathbb N }$ remains open, so the result is given in terms of limit points. In all numerical experiments, these sequences converge, but we can neither prove nor disprove the statement about their convergence. It is worth noting that a similar situation is known for the popular ALS algorithm \cite{mohlenkamp2013musings}.

\begin{algorithm}
\caption{Alternating minimization method.}
\label{alg:alternating_minimization}
\begin{algorithmic}
\Require{matrix $A \in \mathbb{R}^{m \times n}$, rank $r \ge 1$, initial matrix $V^{(0)} \in \mathbb{R}^{n \times r}$.}
\Ensure{factors of rank-$r$ approximation $\hat{U} \in \mathbb{R}^{m \times r}$, $\hat{V} \in \mathbb{R}^{n \times r}$.}

\State{$t = 1$}
\Repeat
\State{$U^{(t)} = \phi(A, V^{(t-1)})$}
\State{$V^{(t)} = \psi(A, U^{(t)})$}

\State{$C = \|U^{(t)}\|_C \|V^{(t)}\|_C$}
\State{$U^{(t)} = U^{(t)} / \|U^{(t)}\|_C \cdot C^{1/2}$}
\State{$V^{(t)} = V^{(t)} / \|V^{(t)}\|_C \cdot C^{1/2}$}
\State{$t = t + 1$}
\Until{convergence}
\State{$\hat{U} = U^{(t - 1)}, \quad \hat{V} = V^{(t - 1)}$}
\end{algorithmic}
\end{algorithm}

\subsection{Rank-$r$ alternance}
\label{sec:rank_r_alternance}

In this section, we provide the properties of the limit points obtained by the alternating minimization method. Let us introduce the necessary notations. Let $A \in \mathbb{R}^{m \times n}$ be a matrix and matrices $U \in \mathbb{R}^{m \times r}$ and $V \in \mathbb{R}^{n \times r}$ be Chebyshev. Let us denote $G = A - U V^T$. We also denote
\begin{equation*}
    S(A, U, V) = \{(i, j): |g_{ij}| = \|G\|_C\},
\end{equation*}
\begin{equation*}
    \mathcal{I}(A, U, V) = \{i: \exists j \text{ such that } (i, j) \in S(A, U, V) \},
\end{equation*}
\begin{equation*}
    \mathcal{J}(A, U, V) = \{j: \exists i \text{ such that } (i, j) \in S(A, U, V) \}.
\end{equation*}

\begin{definition}
\label{definition:alternance_rank_r}
    Let $A \in \mathbb{R}^{m \times n}$ be a matrix and matrices $U \in \mathbb{R}^{m \times r}$ and $V \in \mathbb{R}^{n \times r}$ be Chebyshev. We say that the triple $(A, U, V)$ possesses \textit{a $2$-way alternance of rank $r$}, if there is a non-empty set $\mathcal{A} \subset \{1, \dots, m\} \times \{1, \dots, n\}$ such that $\mathcal{A} \subset S(T, U, V)$ and if $(i, j) \in \mathcal{A}$, then there exist a set $I$ of $r+1$ different indices
    \begin{equation*}
        1 \le i_1 < i_2 < \dots < i_{r+1} \le m
    \end{equation*}
    such that $i \in I$ and a set $J$ of $r+1$ different indices
    \begin{equation*}
        1 \le j_1 < j_2 < \dots < j_{r+1} \le n
    \end{equation*}
    such that $j \in J$ with the following properties.
    \begin{enumerate}
        \item It holds
        \begin{equation*}
            (i, j_1), (i, j_2), \dots, (i, j_{r+1}), (i_1, j), (i_2, j), \dots, (i_{r+1}, j) \in \mathcal{A}.
        \end{equation*}
    
        \item The signs in the sequence
    \begin{equation*}
        g_{i j_1} D_1(\mathcal{V}), g_{i j_2} D_2(\mathcal{V}), \dots, g_{i j_{r+1}} D_{r+1}(\mathcal{V})
    \end{equation*}
    and the signs in the sequence
    \begin{equation*}
        g_{i_1 j} D_1(\mathcal{U}), g_{i_2 j} D_2(\mathcal{U}), \dots, g_{i_{r+1} j} D_{r+1}(\mathcal{U})
    \end{equation*}
    alternate, where $\mathcal{U} = U(I)$ and $\mathcal{V} = V(J)$.
    \end{enumerate}
\end{definition}

The main results of this section are
\begin{enumerate}
    \item if $\hat{U} \in \mathbb{R}^{m\times r}$ and $\hat{V} \in \mathbb{R}^{n\times r}$ are such that $\hat{U}\hat{V}^T$ is the best rank-$r$ approximation to the matrix $A$ in the Chebyshev norm, then $(A, \hat{U}, \hat{V})$ possesses a 2-way alternance of rank $r$;
    \item the limit points of the alternating minimization method possess a 2-way alternance of rank $r$.
\end{enumerate}

The rigorous formulations are as follows.
\begin{theorem}
\label{theorem:necessary_cond}
\label{THEOREM:NECESSARY_COND}
    Let $A \in \mathbb{R}^{m \times n}$ be a matrix, $\rank A > r$. Let $\hat{U} \in \mathbb{R}^{m\times r}$ and $\hat{V} \in \mathbb{R}^{n\times r}$ be a solution to the problem $\|A - UV^T\|_C \to \min\limits_{U \in \mathbb{R}^{m\times r}, V \in \mathbb{R}^{n\times r}}$.
    Let $\hat{V}$ be Chebyshev and the alternating minimization method for the matrix $A$ and the initial point $V^{(0)} = \hat{V}$ be correct. Then the triple $(A, \hat{U}, \hat{V})$ possesses a 2-way alternance of rank $r$.
\end{theorem}

\begin{theorem}
\label{theorem:2d_alternance}
\label{THEOREM:2D_ALTERNANCE}
    Let $A \in \mathbb{R}^{m \times n}$ be a matrix, $\rank A > r$ and the matrix $V \in \mathbb{R}^{n\times r}$ be Chebyshev. Let the alternating minimization method for the matrix $A$ and the initial point $V^{(0)} = V$ be correct and the sequences $\{U^{(t)}\}_{t \in \mathbb N}$ and $\{V^{(t)}\}_{t \in \mathbb N}$ be constructed by the alternating minimization method. Let a limit point $\Xi$ of the sequence $\Xi_t$, where $\Xi_t = V^{(t)} / \|V^{(t)}\|_C$ be Chebyshev. Then $(A, \phi(A, \Xi), \Xi)$ possesses a 2-way alternance of rank $r$.
\end{theorem}

Since the proofs of these theorems are rather technical, we place them in \sect{appendix_theorems}.

\section{Numerical evaluation}
\label{sec:numerical}
In this section, we numerically investigate the properties and effectiveness of the proposed method. The implementation of our algorithm is available online\footnote{\url{https://github.com/stanis-morozov/cheburaxa}}. We also compare our approach with the alternating projections method \cite{budzinskiy2023quasioptimal} for building low-rank approximations in the Chebyshev norm.
The method starts from a random rank-$r$ matrix and alternately projects to the $\varepsilon$--ball in the Chebyshev norm with the center being the target matrix and to the set of low-rank matrices via SVD. With this approach, the method tries to find an element in the intersection of the mentioned sets. Using the binary search the method estimates the approximation error $\varepsilon$.
For alternating projections we use the implementation and parameters provided by the authors.

\begin{figure}
  \begin{subfigure}{0.31\textwidth}
    \includegraphics[width=\linewidth]{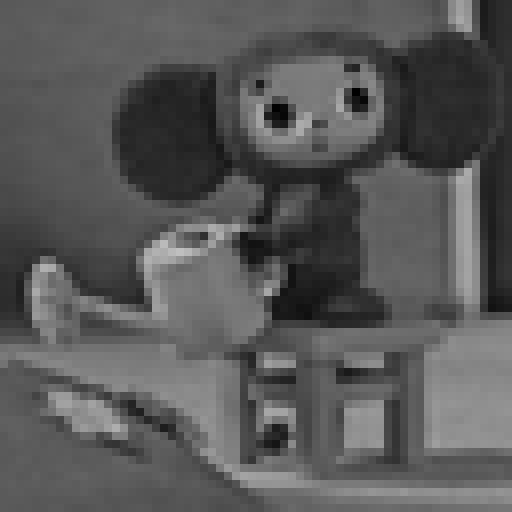}
    \caption{Original image} \label{fig:grayscale_alternance_a}
  \end{subfigure}%
  \hspace*{\fill}   
  \begin{subfigure}{0.31\textwidth}
    \includegraphics[width=\linewidth]{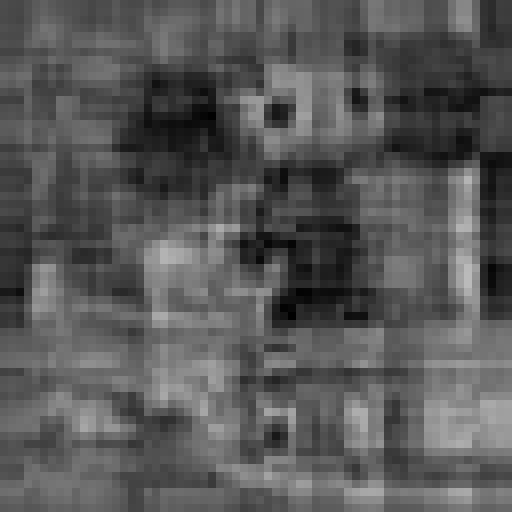}
    \caption{Approximation} \label{fig:grayscale_alternance_b}
  \end{subfigure}%
  \hspace*{\fill}   
  \begin{subfigure}{0.31\textwidth}
    \includegraphics[width=\linewidth]{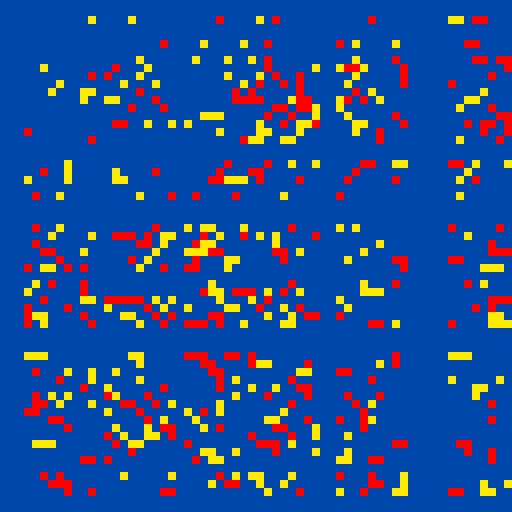}
    \caption{Alternance set} \label{fig:grayscale_alternance_c}
  \end{subfigure}

\caption{An example of low-rank approximation of a grayscale image. The left image corresponds to the original picture of size $64\times 64$, the middle image contains the approximation of rank $8$, the right image demonstrates the 2-way alternance of rank $r$. Blue pixels correspond to the positions where the maximum absolute values in the residual is not reached, yellow where it is reached with the positive signs and red where it is reached with the negative sign.} \label{fig:grayscale_alternance}
\end{figure}

\begin{figure}
     \centering
     \includegraphics[width=0.49\textwidth]{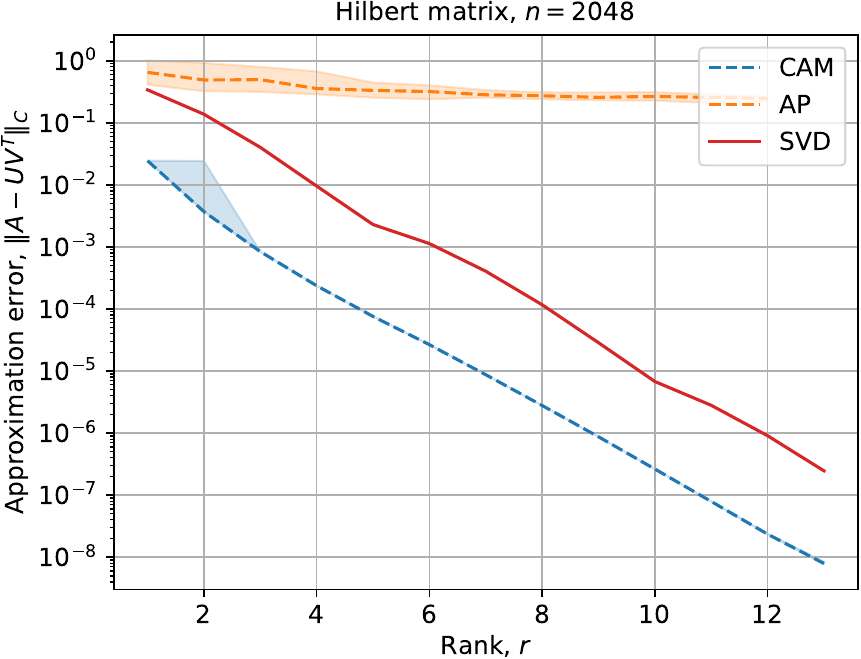}
     \hfill
     \includegraphics[width=0.49\textwidth]{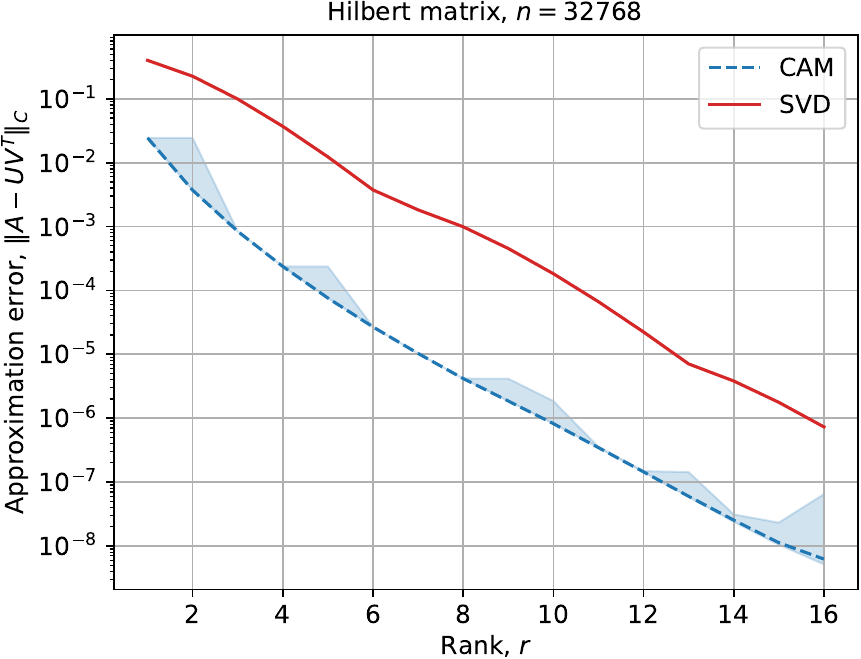}
    \caption{Approximation error of the Hilbert matrix with alternating minimization method (CAM), alternating projections (AP) and SVD. The left panel corresponds to the matrix of size $n=2,048$. The iterative procedures were started from $20$ random initial points. The colored domains correspond to the maximal and minimal values over the initial points and the dotted curves correspond to the median value. The right panel contains the results for the matrix of size $n=32,768$.}
    \label{fig:hilbert_results}
\end{figure}

\subsection{2-way alternance of rank $r$}
To illustrate the result of \thm{2d_alternance}, we run the alternating minimization algorithm for a grayscale image of size $64\times 64$. We represent the grayscale image as the matrix, where the elements are real values from 0 to 1. \fig{grayscale_alternance_a} demonstrates the initial image. Let $U$ and $V$ be the matrices constructed by the alternating minimization method for rank $8$. \fig{grayscale_alternance_b} shows $UV^T$ being the approximation of the image. We remind from \sect{rank_r_alternance} that $G = A - U V^T$ and $S(A, U, V) = \{(i, j): |g_{ij}| = \|G\|_C\}$. \fig{grayscale_alternance_c} shows the set $S(A, U, V)$, namely, the pixel in the position $(i, j)$ is yellow if $g_{ij} = \|G\|_C$, red if $g_{ij} = -\|G\|_C$ and blue if $(i, j) \notin S(A, U, V)$. The set $S(A, U, V)$ corresponds to the 2-way alternance of rank $r$. Note that every row and every column of the matrix in \fig{grayscale_alternance_c} contains either zero elements from the set $S(A, U, V)$, or at least $9$ elements. The signs do not alternate since we plot only the signs of $g_{ij}$ without the determinants (see \dfn{alternance_rank_r}). Note, however, that some rows and columns contain more than $9$ elements from the set $S(A, U, V)$, so we cannot choose the matrices $\mathcal{U}$ and $\mathcal{V}$ from \dfn{alternance_rank_r} that provide the alternance of signs for all elements in such rows and columns. But we can choose the corresponding matrices $\mathcal{U}$ and $\mathcal{V}$ for every subset of size $9$.

\begin{figure}
     \centering
     \includegraphics[width=0.49\textwidth]{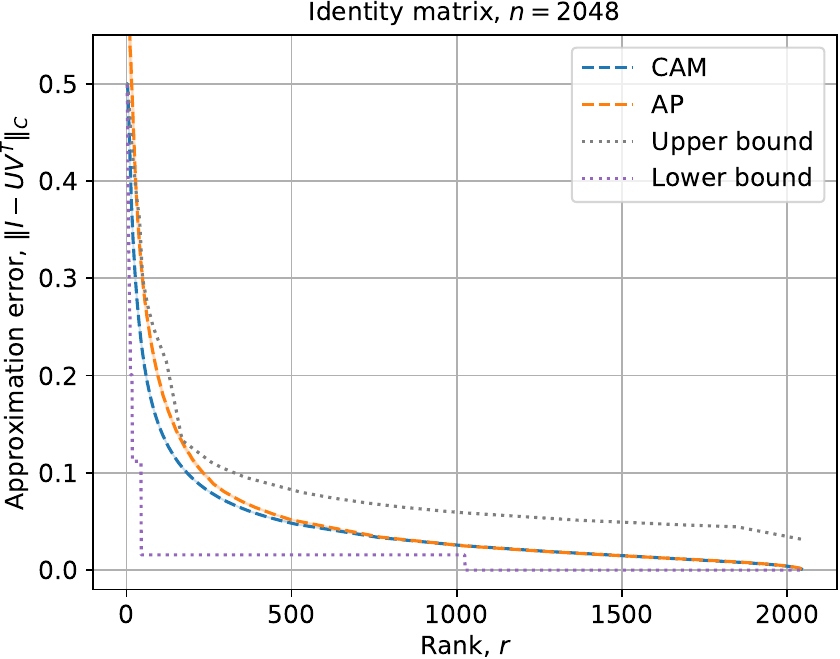}
     \hfill
     \includegraphics[width=0.49\textwidth]{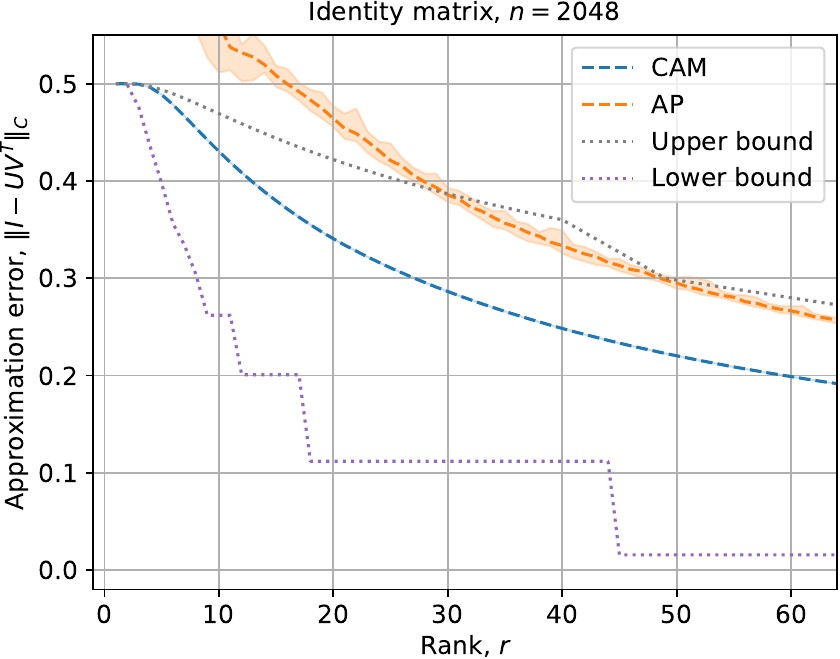}\\
    \caption{Approximation error of the identity matrix with alternating minimization method (CAM) and alternating projections (AP). The matrix size is $n=2,048$. The right plot corresponds to the zoom for the small ranks. The iterative procedures were started from $20$ random initial points. The colored domains correspond to the maximal and minimal values over the initial points and the dotted curves correspond to the median value. The plots also contain the known theoretical upper and lower bounds.}
    \label{fig:identity_results}
\end{figure}

\begin{table}
\centering
\begin{tabular}{|c| c c c c c c c c|} 
 \hline
 Accuracy, $\varepsilon$ & $128$ & $256$ & $512$ & $1,024$ & $2,048$ & $4,096$ & $8,192$ & $16,384$ \\
 \hline
 0.45 & 6  & 6  & 7   & 8   & 9   & 10  & 11  & 12 \\ 
 0.4  & 8  & 9  & 10  & 12  & 13  & 15  & 17  & 18 \\
 0.25 & 17 & 22 & 27  & 33  & 40  & 47  & 54  & 62 \\
 0.1  & 60 & 84 & 112 & 145 & 184 & 228 & 278 & 333 \\
 \hline
\end{tabular}
\caption{The minimal rank needed for the alternating minimization method to achieve the approximation accuracy $\varepsilon$ for the identity matrices of different sizes.}
\label{table:identity_acc_ranks}
\end{table}

\subsection{Hilbert matrix}

The matrix $H = \begin{bmatrix}
    \dfrac{1}{i+j}
\end{bmatrix}_{i,j=1}^{n}$ is known to have exponentially fast decaying singular values, which follows that $H$ can be efficiently approximated by low-rank matrices via SVD. \fig{hilbert_results} shows the results for the matrices of sizes $n=2,048$ and $n=32,768$ and different ranks. We compare the approximation quality in the Chebyshev norm obtained by the proposed Chebyshev alternating minimization (CAM), SVD and alternating projections (AP). For alternating minimization and alternating projections we run the iterative procedures from $20$ random points. To investigate the stability to initial points we draw the maximum and minimum approximation results among $20$ initial points and colorize the domain between them. We also draw the median value with the dotted line. One can see the uniform superiority of the alternating minimization method over the baselines both in terms of accuracy and stability. This experiment demonstrates that even for matrices with rapidly decaying singular values the Chebyshev approximation can be much better than that provided by SVD. For the matrix of size $n=32,768$ we do not provide the results for the alternating projection method since the computation time becomes infeasible.

\begin{figure}
     \centering
     \includegraphics[width=0.49\textwidth]{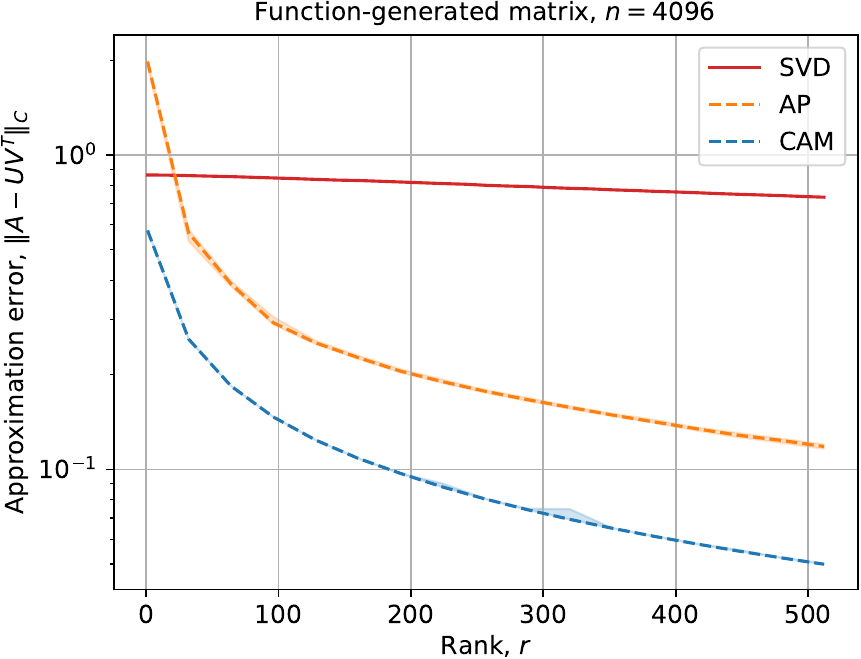}
    \caption{Approximation error for the function-generated matrix with alternating minimization method (CAM), alternating projections (AP) and SVD. The matrix size is $n=4,096$. The iterative procedures were started from $20$ random initial points. The colored domains correspond to the maximal and minimal values over the initial points and the dotted curves correspond to the median value.}
    \label{fig:func_gen}
\end{figure}

\subsection{Identity matrix}

In this experiment, we try to approximate the identity matrix in the Chebyshev norm. \fig{identity_results} shows the results for the matrix of size $n=2,048$. As in the previous experiment, we run the methods from $20$ random initial points and draw the maximum, minimum and median values. We also plot the known in literature upper and lower bounds on the uniform approximation of the identity matrix (see \cite{kashin1975poperechnikah, pinkus2012n, shannon1959probability} for upper bounds and \cite{garnaev1984poperechnikah, gluskin1986octahedron} for lower bounds). One can see that the results of the alternating minimization method are perfectly between the lower and upper bounds for all ranks. One can also see that alternating minimization provides the better approximation than alternating projections again, especially for low ranks, that are usually more practical.

\tab{identity_acc_ranks} shows the values of minimal rank needed to reach the approximation accuracy $\varepsilon$ with the alternating minimization method for $\varepsilon \in \{0.45, 0.4, 0.25, 0.1\}$ for different sizes of the identity matrix. The experiment gives an idea of the accuracy asymptotic of the low-rank approximation of the identity matrix and demonstrates the scalability of the proposed algorithm.

\subsection{Function-generated matrices}

It is shown in \cite{udell2019big} and \cite{budzinskiy2024big} that certain classes of function-generated matrices can be well approximated by low-rank matrices in the Chebyshev norm. In particular, it is shown that if $\{x_j\}_{j=1}^{n}$ are randomly uniformly distributed on the $d$-dimensional sphere, then the matrix $A \in \mathbb{R}^{n\times n}$, where
\begin{equation}
\label{eq:gaussian_kernel_matrix}
    a_{ij} = \exp{\left(-\|x_i - x_j\|_2^2 \right)},
\end{equation}
admits an accurate low-rank approximation. In this experiment, we generate $n=4,096$ random points from the uniform distribution on the $d$-dimensional sphere, where $d=8,192$, and compute the matrix $A$ according to the formulas \eq{gaussian_kernel_matrix}. We compare the approximation quality in the Chebyshev norm obtained by the Chebyshev alternating minimization method (CAM), SVD and alternating projections (AP) in \fig{func_gen}. For alternating minimization and alternating projections we again run the iterative procedure from $20$ random initial points and draw the maximum, minimum and median values. The matrix $A$ is the same for all methods. One can see in \fig{func_gen} that alternating minimization method again outperforms the baselines both in accuracy and stability.

\section{Conclusion}
\label{sec:conclusion}
In this paper, we propose the accelerated algorithm for solving the best uniform approximation problem. Via the numerical evaluation, we demonstrate the effectiveness of the accelerated alternating minimization method and its superiority over the alternating projections method in terms of accuracy, stability and computation time. We also propose a convenient equioscillation criterion for a vector to be the solution to the best uniform approximation problem. Finally, we introduce the notion of a $2$-way alternance of rank $r$ and demonstrate that it constitutes the necessary condition of the best low-rank approximation in the Chebyshev norm. We study the alternating minimization algorithm theoretically and show that this necessary condition is satisfied by all limits points of the algorithm.

\bibliography{wileyNJD-AMA}

\appendix

\bmsection{Proof of THEOREM~\ref{theorem:necessary_cond} and THEOREM~\ref{theorem:2d_alternance}}
\label{sec:appendix_theorems}
To prove the theorems we need several technical lemmas.

\begin{lemma}\label{lemma:alt_basic1}
    Let $A \in \mathbb{R}^{m \times n}$, $\rank A > r$ and matrices $V \in \mathbb{R}^{n\times r}$, $U = \phi(A, V)$ be Chebyshev. Then for all $ i \in \mathcal{I}(A, U, V)$ there exists a set $J = (j_1, \dots, j_{r+1})$, where
    \begin{equation*}
        1 \le j_1 < j_2 < \dots < j_{r+1} \le n,
    \end{equation*}
    such that $(i, j_1), \dots, (i, j_{r+1}) \in S(A, U, V)$ and the signs in the sequence
    \begin{equation*}
        g_{i j_1} D_1(\mathcal{V}), g_{i j_2} D_2(\mathcal{V}), \dots, g_{i j_{r+1}} D_{r+1}(\mathcal{V})
    \end{equation*}
    alternate, where $\mathcal{V} = V(J)$.
\begin{proof}
    Let $ i \in \mathcal{I}(A, U, V)$. By the definition of $\phi$, the $i$--th row of the matrix $U$ is the solution to the problem
    \begin{equation*}
        \|a^i - V x\|_\infty \to \min\limits_{x \in \mathbb{R}^r},
    \end{equation*}
    whence by \thm{residual_alternating_signs} there is a set of $r+1$ integers
    \begin{equation*}
        1 \le j_1 < j_2 < \dots < j_{r+1} \le n
    \end{equation*}
    such that the signs in the sequence
    \begin{equation*}
        (a_{i j_1} - (v^{j_1})^T u^{i}) D_1(\mathcal{V}), (a_{i j_2} - (v^{j_2})^T u^{i}) D_2(\mathcal{V}), \dots, (a_{i j_{r+1}} - (v^{j_{r+1}})^T u^{i}) D_1(\mathcal{V})
    \end{equation*}
    alternate and
    \begin{equation*}
        |a_{i j_1} - (v^{j_1})^T u^{i}| = \dots = |a_{i j_{r+1}} - (v^{j_{r+1}})^T u^{i}| = \|a^i - V u^{i}\|_\infty = \|G\|_C,
    \end{equation*}
    where the last equality is due to $ i \in \mathcal{I}(A, U, V)$.
\end{proof}
\end{lemma}

\begin{lemma}\label{lemma:alt_basic2}
    Let $A \in \mathbb{R}^{m \times n}$, $\rank A > r$ and matrices $V \in \mathbb{R}^{n\times r}$, $U = \phi(A, V)$, $\tilde{V} = \psi(A, U)$
    be Chebyshev.
    Let also $\|A - UV^T\|_C = \|A - U\tilde{V}^T\|_C$. Then $S(A, U, \tilde{V}) \subset S(A, U, V)$. Moreover, the following conditions are equivalent
    \begin{enumerate}
        \item $j \in \mathcal{J}(A, U, \tilde{V})$;
        \item $j \in \mathcal{J}(A, U, V)$ and $v^j = \tilde{v}^j$;
        \item there exists a set $I = (i_1, \dots, i_{r+1})$ such that
    \begin{equation*}
        1 \le i_1 < i_2 < \dots < i_{r+1} \le m,
    \end{equation*}
    $(i_1, j), \dots, (i_{r+1}, j) \in S(A, U, V)$ and the signs in the sequence
    \begin{equation*}
        g_{i_1 j} D_1(\mathcal{U}), g_{i_2 j} D_2(\mathcal{U}), \dots, g_{i_{r+1} j} D_{r+1}(\mathcal{U})
    \end{equation*}
    alternate, where $\mathcal{U} = U(I)$.
    \end{enumerate}
\begin{proof}
    Consider an index $j$. Since $\tilde{V} = \psi(A, U)$,
    \begin{equation*}
        \| a_j - U v^j \|_\infty \ge \| a_j - U \tilde{v}^j \|_\infty.
    \end{equation*}
    Let $j \notin \mathcal{J}(A, U, V)$. Then $\| a_j - U v^j \|_\infty < \|A - U V^T\|_C$,
    whence $\| a_j - Uv^j \|_\infty < \|A - U \tilde{V}^T\|_C$,
    therefore $j \notin \mathcal{J}(A, U, \tilde{V})$.

    Now let $\tilde{v}^j \neq v^j$. Due to the uniqueness of the uniform approximation problem (see \thm{exists_unique_cont}),
    \begin{equation*}
        \| a_j - U \tilde{v}^j \|_\infty < \| a_j - U v^j \|_\infty \le
        \|A - UV^T\|_C = \|A - U \tilde{V}^T\|_C,
    \end{equation*}
    whence follows $j \notin \mathcal{J}(A, U, \tilde{V})$. Thus we have proven that if $j \in \mathcal{J}(A, U, \tilde{V})$, then $j \in \mathcal(A, U, V)$ and $\tilde{v}^j = v^j$.
    
    Let us prove the reversal, let $j \in \mathcal{J}(A, U, V)$ and $\tilde{v}^j = v^j$. From the latter we have
    \begin{equation*}
        \| a_j - U v^j \|_\infty = \| a - U \tilde{v}^j\|_\infty,
    \end{equation*}
    and from $j \in \mathcal{J}(A, U, V)$ we obtain
    \begin{equation*}
        \| a_j - U v^j \|_\infty = \|A - UV^T \|_C = \|A - U \tilde{V}^T\|_C.
    \end{equation*}
    Therefore, $j \in \mathcal{J}(A, U, \tilde{V})$.

    Since $\tilde{v}^j$ is the solution to the problem
    \begin{equation*}
    \| a_j - Ux \|_\infty \to \min\limits_{x \in \mathbb{R}^r},
    \end{equation*}
    by \thm{residual_alternating_signs} there exists a set $I = (i_1, i_2, \dots, i_{r+1})$, where
    \begin{equation*}
    1 \le i_1 < i_2 < \dots < i_{r+1} \le m,
    \end{equation*}
    such that
    \begin{equation*}
        |a_{i_1 j} - (\tilde{v}^{j})^T u^{i_1}| = |a_{i_1 j} - (\tilde{v}^{j})^T u^{i_1}| = \| a_j - U \tilde{v}^j\|_{\infty},
    \end{equation*}
    and the signs in the sequence
    \begin{equation*}
        (a_{i_1 j} - (\tilde{v}^{j})^T u^{i_1}) D_1(\mathcal{U}), (a_{i_2 j} - (\tilde{v}^{j})^T u^{i_2}) D_2(\mathcal{U}), \dots, (a_{i_{r+1} j} - (\tilde{v}^{j})^T u^{i_{r+1}}) D_{r+1}(\mathcal{U})
    \end{equation*}
    alternate, where $\mathcal{U} = U(I)$. Since $j \in \mathcal{J}(A, U, \tilde{V})$,
    \begin{equation*}
        \| a_j - U \tilde{v}^j\|_{\infty} = \|A - U\tilde{V}^T\|_C = \|A - UV^T\|_C,
    \end{equation*}
    therefore $(i_1, j), (i_2, j), \dots, (i_{r+1}, j) \in S(A, U, \tilde{V})$, but due to $v^j = \tilde{v}^j$,
    \begin{equation*}
        (i_1, j), (i_2, j), \dots, (i_{r+1}, j) \in S(A, U, V)
    \end{equation*}
    and the signs is the sequence
    \begin{equation*}
        (a_{i_1 j} - (v^{j})^T u^{i_1}) D_1(\mathcal{U}), (a_{i_2 j} - (v^{j})^T u^{i_2}) D_2(\mathcal{U}), \dots, (a_{i_{r+1} j} - (v^{j})^T u^{i_{r+1}}) D_{r+1}(\mathcal{U})
    \end{equation*}
    alternate.

    Conversely, let there is a set $I = (i_1, i_2, \dots, i_{r+1})$, where
    \begin{equation*}
    1 \le i_1 < i_2 < \dots < i_{r+1} \le m,
    \end{equation*}
    such that $(i_1, j), \dots, (i_{r+1}, j) \in S(A, U, V)$ and the signs in the sequence
    \begin{equation*}
        g_{i_1 j} D_1(\mathcal{U}), g_{i_2 j} D_2(\mathcal{U}), \dots, g_{i_{r+1} j} D_{r+1}(\mathcal{U})
    \end{equation*}
    alternate, where $\mathcal{U} = U(I)$. Then, by \thm{residual_alternating_signs}, $v^j$ is the solution to the problem
    \begin{equation*}
    \| a_j - Ux \|_\infty \to \min\limits_{x \in \mathbb{R}^r},
    \end{equation*}
    but by \thm{exists_unique_cont} the solution is unique, hence $v^j = \tilde{v}^j$ and $j \in \mathcal{J}(A, U, V)$.

    Let us finally prove that $S(A, U, \tilde{V}) \subset S(A, U, V)$. Let a pair $(i, j) \in S(A, U, \tilde{V})$. Then $j \in \mathcal{J}(A, U, \tilde{V})$, whence $\tilde{v}^j = v^j$. Hence,
    \begin{equation*}
        a_{i j} - (v^j)^T u^i = a_{i j} - (\tilde{v}^j)^T u^i.
    \end{equation*}
    However,
    \begin{equation*}
        |a_{i j} - (\tilde{v}^j)^T u^i| = \|A - U \tilde{V}^T\|_C = \|A - UV^T\|_C,
    \end{equation*}
    which follows that $(i, j) \in S(A, U, V)$.
\end{proof}
\end{lemma}

\begin{lemma}
\label{lemma:interchange_sign_alternate}
\label{LEMMA:INTERCHANGE_SIGN_ALTERNATE}
Let a matrix $\hat{V} \in \mathbb{R}^{(r+1)\times r}$ and $h \in \mathbb{R}^{r}$ are such that $\begin{bmatrix}
    \hat{V} \\
    h^T
\end{bmatrix} \in \mathbb{R}^{(r+2)\times r}$ is Chebyshev. Let a vector $\hat{w} \in \mathbb{R}^{r+1}$ have non-zero components and $\xi \in \mathbb{R}$, $\xi \neq 0$. Let the signs in the sequence
\begin{equation*}
    \hat{w}_1 D_1(\hat{V}), \hat{w}_2 D_2(\hat{V}), \dots, \hat{w}_{r+1} D_{r+1}(\hat{V})
\end{equation*}
alternate. Then there is $k \in \{1, \dots, r+1\}$ such that the matrix $\tilde{V}$ is obtained from $\hat{V}$ by the replacement of $k$--th row with the vector $h$ and the vector $\tilde{w}$ is obtained from $\hat{w}$ by the replacement of $k$--th element with $\xi$ and the signs in the sequence
\begin{equation*}
    \tilde{w}_1 D_1(\tilde{V}), \tilde{w}_2 D_2(\tilde{V}), \dots, \tilde{w}_{r+1} D_{r+1}(\tilde{V})
\end{equation*}
alternate.
\end{lemma}
Since the proof of the last lemma is not directly connected with the idea of alternance, we place it in \sect{appendix_proof} in order not to interrupt the thought.

\begin{lemma}\label{lemma:alt_basic3}
    Let $A \in \mathbb{R}^{m \times n}$, $\rank A > r$ and matrices $V \in \mathbb{R}^{n\times r}$, $U = \phi(A, V)$, $\tilde{V} = \psi(A, U)$
    be Chebyshev. Let also
    \begin{equation*}
        \|A - U V^T\|_C = \|A - U \tilde{V}^T\|_C
    \end{equation*}
    and
    \begin{equation*}
        S(A, U, V) = S(A, U, \tilde{V}).
    \end{equation*}
    Then $(A, U, V)$ possesses a 2-way alternance of rank $r$.
\begin{proof}
    Let us show that $(A, U, V)$ possesses a 2-way alternance of rank $r$ with the set of indices $\mathcal{A} = S(A, U, V)$. Let $(i, j) \in \mathcal{A}$, then $i \in \mathcal{I}(A, U, V)$ and by \lem{alt_basic1} there exists set $J = (j_1, \dots, j_{r+1})$, where
    \begin{equation*}
        1 \le j_1 < j_2 < \dots < j_{r+1} \le n
    \end{equation*}
    such that $(i, j_1), \dots, (i, j_{r+1}) \in S(A, U, V)$ and the signs in the sequence
    \begin{equation*}
        g_{i j_1} D_1(\mathcal{V}), g_{i j_2} D_2(\mathcal{V}), \dots, g_{i j_{r+1}} D_{r+1}(\mathcal{V})
    \end{equation*}
    alternate, where $\mathcal{V} = V(J)$. It can happen that $j \notin J$. But due to \lem{interchange_sign_alternate} there is a set $J'$, obtained from $J$ by the replacement of one element with $j$, such that the signs in the sequence
    \begin{equation}
    \label{eq:resulting_alternance}
        g_{i j'_1} D_1(\mathcal{V}'), g_{i j'_2} D_2(\mathcal{V}'), \dots, g_{i j'_{r+1}} D_{r+1}(\mathcal{V}')
    \end{equation}
    alternate, where $J' = (j'_1, \dots, j'_{r+1})$ and $\mathcal{V}' = V(J')$. The set $J'$ can be not ordered, but note that the swap of two neighboring elements in $J'$ preserves the alternance in \eq{resulting_alternance}. Therefore, there is a set $J' = (j'_1, \dots, j'_{r+1})$, where
    \begin{equation*}
        1 \le j'_1 < j'_2 < \dots < j'_{r+1} \le n,
    \end{equation*}
    such that $j \in J'$, $(i, j'_1), \dots, (i, j'_{r+1}) \in S(A, U, V)$ and the signs in \eq{resulting_alternance} alternate.

    Let $(i, j) \in S(A, U, \tilde{V}) = S(A, U, V) = \mathcal{A}$. Then $j \in \mathcal{J}(A, U, \tilde{V})$, whence from \lem{alt_basic2} there exists a set $I = (i_1, \dots, i_{r+1})$ such that
    \begin{equation*}
        1 \le i_1 < i_2 < \dots < i_{r+1} \le m,
    \end{equation*}
    $(i_1, j), \dots, (i_{r+1}, j) \in S(A, U, V)$ and the signs in the sequence
    \begin{equation*}
        g_{i_1 j} D_1(\mathcal{U}), g_{i_2 j} D_2(\mathcal{U}), \dots, g_{i_{r+1} j} D_{r+1}(\mathcal{U})
    \end{equation*}
    alternate, where $\mathcal{U} = U(I)$. If $i \notin I$, we can repeat the reasoning above and derive that there is a set $I' = (i'_1, \dots, i'_{r+1})$, where
    \begin{equation*}
        1 \le i'_1 < i'_2 < \dots < i'_{r+1} \le m,
    \end{equation*}
    such that $i \in I'$, $(i'_1, j), \dots, (i'_{r+1}, j) \in S(A, U, V)$ and
    \begin{equation*}
        g_{i'_1 j} D_1(\mathcal{U}'), g_{i'_2 j} D_2(\mathcal{U}'), \dots, g_{i'_{r+1} j} D_{r+1}(\mathcal{U}')
    \end{equation*}
    alternate, where $\mathcal{U}' = U(I')$.
    Therefore, all properties from the 2-way alternance of rank $r$ definition are fulfilled.
\end{proof}
\end{lemma}

\begin{lemma}\label{lemma:alt_basic4}
    Let $A \in \mathbb{R}^{m \times n}$, $\rank A > r$ and matrices $V \in \mathbb{R}^{n\times r}$, $U = \phi(A, V)$, $\tilde{V} = \psi(A, U)$
    , $\tilde{U} = \phi(A, \tilde{V})$
    be Chebyshev. Let also $\|A - \tilde{U}\tilde{V}^T\|_C = \|A - UV^T\|_C$ and $(A, \tilde{U}, \tilde{V})$ possesses a 2-way alternance of rank $r$. Then $(A, U, V)$ also possesses a 2-way alternance of rank $r$.
\begin{proof}

    By \lem{basic_prop}~\ref{basic_prop_i}
    \begin{equation*}
        \|A - UV^T\|_C \ge \|A - U \tilde{V}^T\|_C \ge \|A - \tilde{U} \tilde{V}^T\|_C,
    \end{equation*}
    however, the first and last terms are equal, hence
    \begin{equation*}
        \|A - UV^T\|_C = \|A - U \tilde{V}^T\|_C = \|A - \tilde{U} \tilde{V}^T\|_C.
    \end{equation*}
    Applying \lem{alt_basic2}, we get $S(A, U, \tilde{V}) \subset S(T, U, V)$.
    Since $\|A^T - \tilde{V}U^T\|_C = \|A^T - \tilde{V} \tilde{U}^T\|_C$, we have $S(A^T, \tilde{V}, \tilde{U}) \subset S(A^T, \tilde{V}, U)$, which is equivalent to $S(A, \tilde{U}, \tilde{V}) \subset S(A, U, \tilde{V})$, hence $S(A, \tilde{U}, \tilde{V}) \subset S(A, U, V)$.

    Let $(i, j)$ belongs to the 2-way alternance of rank $r$ for a triple $(A, \tilde{U}, \tilde{V})$ and $I=(i_1, \dots, i_{r+1})$ and $J=(j_1, \dots, j_{r+1})$ be the sets of indices from the definition of the alternance. From \lem{alt_basic2} we have $u^i = \tilde{u}^i$ for all $i \in \mathcal{I}(A, \tilde{U}, \tilde{V})$, hence $U(I) = \tilde{U}(I)$. Similarly, we have $V(J) = \tilde{V}(J)$. Therefore,
    \begin{equation*}
        (a_{i_k j} - (u^{i_k})^T v^j) D_k(U(I)) = (a_{i_k j} - (\tilde{u}^{i_k})^T v^j) D_k(\tilde{U}(I)), \quad k = 1, 2, \dots, r + 1
    \end{equation*}
    since $j \in J$. Also we have,
    \begin{equation*}
        (a_{i j_k} - (u^{i})^T v^{j_k}) D_k(V(J)) = (a_{i j_k} - (u^{i})^T \tilde{v}^{j_k}) D_k(\tilde{V}(J)), \quad k = 1, 2, \dots, r + 1
    \end{equation*}
    since $i \in I$.
    Therefore, a 2-way alternance of rank $r$ for $(A, \tilde{U}, \tilde{V})$ is the 2-way alternance of rank $r$ for $(A, U, V)$.
\end{proof}
\end{lemma}

\begin{lemma}\label{lemma:existence_alternance1}
    Let $A \in \mathbb{R}^{m \times n}$, $\rank A > r$ and a matrix $V \in \mathbb{R}^{n\times r}$ be Chebyshev. Let the sequences $\{U^{(t)}\}_{t \in \mathbb N}$ and $\{V^{(t)}\}_{t \in \mathbb N}$ be Chebyshev and constructed by the alternating minimization method for the matrix $A$ and the initial point $V^{(0)} = V$. Let also $\Xi \in \mathbb{R}^{n \times r}$ be a limit point of the sequence $\Xi^{(t)} = V^{(t)} / \|V^{(t)}\|_C$ and Chebyshev. Then
    \begin{equation*}
            E(A,\Xi) = \|A - \phi(A,\Xi) \Xi^T\|_C = E(A,V).
    \end{equation*}

    \begin{proof}        
        \lem{basic_prop_E}~\ref{basic_prop_E_ii} implies $E(T, \Xi^{(t)}) = E(A,V)$ for all $t \in \mathbb N$. Due to the upper semi-continuity of $E$ (see \lem{basic_prop_E}~\ref{basic_prop_E_iii}) we have $E(A,\Xi) \ge E(A,V)$. Moreover, $\|A - \phi(A,\Xi)\Xi^T\|_C \ge E(A,\Xi)$.
        \begin{multline*}
            \|A - \phi(A,\Xi)\Xi^T\|_C = \lim\limits_{t \to \infty} \|A - \phi(A,\Xi^{(l_t)})(\Xi^{(l_t)})^T\|_C =\\
            \lim\limits_{t \to \infty} \|A - \phi(A, V^{(l_t)}) (V^{(l_t)})^T\|_C = E(A,V).
        \end{multline*}
        Therefore,
        \begin{equation*}
            E(A,V) = \|A - \phi(A,\Xi)\Xi^T\|_C \ge E(A,\Xi) \ge E(A,V)
        \end{equation*}
        and the lemma is proven.

    \end{proof}
\end{lemma}

\begin{lemma}\label{lemma:existence_alternance2}
    Let $A \in \mathbb{R}^{m \times n}$ be a matrix, $\rank A > r$ and a matrix $V \in \mathbb{R}^{n\times r}$ be Chebyshev. Let the alternating minimization method for the matrix $A$ and the initial point $V$ be correct. Then if $\|A - \phi(A,V) V^T\|_C = E(A,V)$, then $(A,\phi(A,V),V)$ possesses a 2-way alternance of rank $r$.
    \begin{proof}

        Let the pair of sequences $\{U^{(t)}\}_{t \in \mathbb N}$ and $\{V^{(t)}\}_{t \in \mathbb N}$ be constructed by the alternating minimization method for $A$ and the initial point $V^{(0)} = V$. Clearly,
        \begin{equation*}
            \| A - U^{(t)} (V^{(t-1)})^T \|_C = \| A - U^{(t)} (V^{(t)})^T \|_C = \\
            \| A - U^{(t+1)} (V^{(t)})^T \|_C
        \end{equation*}
        for all $t \in \mathbb N$. Thus, \lem{alt_basic2} implies that
        \begin{equation*}
            S(A, U^{(1)}, V^{(0)}) \supset S(A, U^{(1)}, V^{(1)}) \supset S(A, U^{(2)}, V^{(1)}) \supset S(A, U^{(2)}, V^{(2)}) \supset \dots
        \end{equation*}
        Since all sets in this sequence are finite and non-empty, there is  $t \in \mathbb N$  such that $S(A, U^{(t+1)}, V^{(t)}) = S(A, U^{(t+1)}, V^{(t+1)})$.
        \lem{alt_basic3} implies that $(A, U^{(t+1)}, V^{(t)})$ possesses a 2-way alternance of rank $r$. Applying \lem{alt_basic4} $t$ times we obtain that $(A, U^{(1)}, V^{(0)}) = (A, \phi(T,V), V)$ possesses a 2-way alternance of rank $r$.
    \end{proof}
\end{lemma}

\lem{existence_alternance1} implies that if the sequences $\{U^{(t)}\}_{t \in \mathbb N}$ and $\{V^{(t)}\}_{t \in \mathbb N}$ are constructed by the alternating minimization method, then a limit point $\Xi$ of the sequence $\Xi^{(t)} = V^{(t)} / \|V^{(t)}\|_C$ satisfies
\begin{equation*}
    E(A,\Xi) = \|A - \phi(A,\Xi) \Xi^T\|_C = E(A,V^{(0)}).
\end{equation*}
In turn, \lem{existence_alternance2} implies that if this is the case, then $(A,\phi(\Xi,V),\Xi)$ possesses a 2-way alternance of rank $r$. Now we are ready to prove the main results of \sect{rank_r_alternance}.

\begin{proof}[Proof of \thm{necessary_cond}]
    Clearly, $$\|A - \hat{U}\hat{V}^T\|_C \ge \|A - \phi(A, \hat{V}) \hat{V}^T\|_C \ge E(A, \hat{V}) \ge \|A - \hat{U}\hat{V}^T\|_C,$$
    where the last inequality is due to $\hat{U}$ and $\hat{V}$ provide the optimal solution. Therefore, $\|A - \phi(A, \hat{V}) \hat{V}^T\|_C = E(A, \hat{V})$ and by \lem{existence_alternance2} the triple $(A, \phi(A, \hat{V}), \hat{V})$ possesses a 2-way alternance of rank $r$. Note that $\mathcal{I}(A, \phi(A, \hat{V}), \hat{V}) \subset \mathcal{I}(A, \hat{U}, \hat{V})$ by the definition of $\phi$. Let us denote $\tilde{U} = \phi(A, \hat{V})$. For $i \in \mathcal{I}(A, \phi(A, \hat{V}), \hat{V})$ we have
    \begin{equation*}
        \|a^i - \hat{V}\hat{u}^i\|_\infty = \|a^i - \hat{V}\tilde{u}^i\|_\infty = \min\limits_{u \in \mathbb{R}^r} \|a^i - \hat{V} u\|_\infty.
    \end{equation*}
    Due to uniqueness of the solution (see \thm{exists_unique_cont}), $\hat{u}^i = \tilde{u}^i$ for $i \in \mathcal{I}(A, \phi(A, \hat{V}), \hat{V})$. Since in \dfn{alternance_rank_r} we can take into account only rows of the matrix $U$ such that $i \in \mathcal{I}(A, \phi(A, V), V)$, we obtain that the alternance for the triple $(A, \phi(A, \hat{V}), \hat{V})$ is the alternance for the triple $(A, \hat{U}, \hat{V})$.
\end{proof}

\begin{proof}[Proof of \thm{2d_alternance}]
Directly follows from \lem{existence_alternance1} and \lem{existence_alternance2}.
\end{proof}

\thm{necessary_cond} implies that the presence of a 2-way alternance of rank $r$ is the necessary condition of the optimal Chebyshev approximation. In turn, from \thm{2d_alternance} it follows that all limit points of the alternating minimization method satisfy this condition.

\bmsection{Proof of LEMMA~\ref{lemma:interchange_sign_alternate}}
\label{sec:appendix_proof}
In this section, we present the proof for \lem{interchange_sign_alternate}.
First, let us introduce some notations. Let $V \in \mathbb{R}^{(r+2)\times r}$ be a Chebyshev matrix and $w \in \mathbb{R}^{r+2}$ be a vector. We denote
\begin{equation*}
    D_j = \det{V(( 1, \dots, j-1, j+1, \dots, r + 1))},
\end{equation*}
\begin{equation*}
    \hat{D}_j^k = \begin{cases}
        \det{V(\{ 1, 2, \dots, k-1, r+2, k+1, \dots, j-1, j+1, \dots,  r + 1 \})}, & k < j, \\
        \det{V(\{ 1, 2, \dots, j-1, j+1, \dots, k-1, r+2, k+1, \dots,  r + 1 \})}, & k > j, \\
        D_j, & k = j.
    \end{cases}
\end{equation*}
It is easy to see that
\begin{equation*}
    \hat{D}_i^j = (-1)^{i - j + 1} \hat{D}_j^i, \quad i \neq j.
\end{equation*}
We denote by $\hat{w}^k \in \mathbb{R}^{r+1}$ the vector such that
\begin{equation*}
    \hat{w}_j^k = \begin{cases}
    w_j, & j \neq k \\
    w_{r+2}, & j=k.
    \end{cases}
\end{equation*}

\begin{lemma}
\label{lemma:determinants_lemma}
For any pairwise distinct $i$, $j$ and $k$ the conditions
\begin{equation*}
    \sign (\hat{D}_i^k D_i \hat{D}_j^k D_j) = -1\text{ and }\sign (\hat{D}_k^i D_k \hat{D}_j^i D_j) = -1
\end{equation*}
cannot be satisfied simultaneously.
\begin{proof}
Let us assume that $\sign (\hat{D}_i^k D_i \hat{D}_j^k D_j) = -1$ and $\sign (\hat{D}_k^i D_k \hat{D}_j^i D_j) = -1$. Since $\hat{D}_k^i = (-1)^{k - i + 1} \hat{D}_i^k$, we have
\begin{eqnarray*}
    \sign (\hat{D}_i^k D_i \hat{D}_j^k D_j) = -1 \\
    \sign (\hat{D}_i^k D_k \hat{D}_j^i D_j) = (-1)^{k-i}. \\
\end{eqnarray*}
Let $\sign(D_j \hat{D}_i^k) = \delta$, then
\begin{eqnarray*}
\sign(D_i \hat{D}_j^k) = -\delta \\
\sign(D_k \hat{D}_j^i) = (-1)^{k-i}\delta \\
\sign(D_j \hat{D}_i^k) = \delta. \\
\end{eqnarray*}
Let $\mathcal{L}$ be the the linear hull of $v^1, \dots, v^{r+1}$, without $v^i, v^j$ and $v^k$. Since $V$ is Chebyshev, the vector $v^k$ can be uniquely represented as
\begin{equation}
\label{eq:relation_ijk_1}
v^k = \alpha v^i + \gamma v^j + z_1,
\end{equation}
where $z_1 \in \mathcal{L}$. We also represent
\begin{equation}
\label{eq:relation_ijk_2}
v^i = \beta v^k + \tau v^{r+2} + z_2,
\end{equation} 
where $z_2 \in \mathcal{L}$.

From \eq{relation_ijk_1} and the properties of the determinant we obtain that
\begin{equation*}
D_i = \det \begin{bmatrix}
(v^1)^T \\
(v^2)^T \\
\vdots \\
(v^{i-1})^T \\
(v^{i+1})^T \\
\vdots \\
(v^{k-1})^T \\
(v^{k})^T \\
(v^{k+1})^T \\
\vdots \\
(v^{r+1})^T
\end{bmatrix} = \det \begin{bmatrix}
(v^1)^T \\
(v^2)^T \\
\vdots \\
(v^{i-1})^T \\
(v^{i+1})^T \\
\vdots \\
(v^{k-1})^T \\
(\alpha v^i + \gamma v^j + z_1)^T \\
(v^{i+1})^T \\
\vdots \\
(v^{r+1})^T
\end{bmatrix} = \det \begin{bmatrix}
(v^1)^T \\
(v^2)^T \\
\vdots \\
(v^{i-1})^T \\
(v^{i+1})^T \\
\vdots \\
(v^{k-1})^T \\
\alpha (v^i)^T \\
(v^{k+1})^T \\
\vdots \\
(v^{r+1})^T
\end{bmatrix},
\end{equation*}
then $D_i = \alpha (-1)^{k-i+1} D_k$.
Similarly, from \eq{relation_ijk_2} and the properties of the determinant, we can deduce that $\hat{D}_j^k = -\beta \hat{D}_j^i$.
Hence,
\begin{equation*}
D_i \hat{D}_j^k = \alpha (-1)^{k-i+1} D_k (-1) \beta \hat{D}_j^i = \alpha \beta (-1)^{k-i} D_k \hat{D}_j^i.
\end{equation*}
Taking the signs from both sides, we get $-\delta = \sign(\alpha \beta)(-1)^{k-i} (-1)^{k-i} \delta$, whence $\sign(\alpha \beta) = -1$.

Note that the coefficients in \eq{relation_ijk_1} can be expressed by Cramer's formulas, and due to $V$ being the Chebyshev matrix, we have $\alpha, \gamma \neq 0$. Then we can express the vector $v^j$ from \eq{relation_ijk_1} as $v^j = \dfrac{1}{\gamma}v^k - \dfrac{\alpha}{\gamma} v^i - \dfrac{1}{\gamma}z_1$.
Then
\begin{equation}
\label{eq:relation_ijk_4}
D_i = -\dfrac{\alpha}{\gamma}(-1)^{j-i+1}D_j = \dfrac{\alpha}{\gamma}(-1)^{j-i}D_j.
\end{equation}

From \eq{relation_ijk_1} and \eq{relation_ijk_2} we have $v^i = \beta(\alpha v^i + \gamma v^j + z_1) + \tau v^{r+2} + z_2$,
hence
\begin{equation}
\label{eq:relation_ijk_3}
(1 - \alpha \beta) v^i = \beta\gamma v^j + \beta z_1 + \tau v^{r+2} + z_2.
\end{equation}
Since $\sign(\alpha \beta) = -1$, we have $1 - \alpha\beta \neq 0$. Then
\begin{equation*}
v^i = \dfrac{\beta\gamma}{1 - \alpha \beta} v^j + \dfrac{\beta z_1 + \tau v^{r+2} + z_2}{1 - \alpha \beta}.
\end{equation*}
and we again derive from the properties of the determinant that
\begin{equation}
\label{eq:relation_ijk_5}
\hat{D}_j^k = \dfrac{\beta\gamma}{1 - \alpha \beta} (-1)^{j-i+1}\hat{D}_i^k,
\end{equation}
so we have from \eq{relation_ijk_4} and \eq{relation_ijk_5} that
\begin{equation*}
D_i \hat{D}_j^k = \dfrac{\alpha}{\gamma}(-1)^{j-i}D_j \dfrac{\beta\gamma}{1 - \alpha \beta} (-1)^{j-i+1}\hat{D}_i^k = \dfrac{\alpha\beta}{1 - \alpha \beta} (-1) D_j \hat{D}_i^k.
\end{equation*}
Taking the signs from both sides we get that
\begin{equation*}
-\delta = \dfrac{\sign(\alpha \beta)}{\sign(1 - \alpha\beta)}(-1)\delta,
\end{equation*}
therefore, $\sign(\alpha\beta) = \sign(1 - \alpha\beta) = -1$,
which is the contradiction.
\end{proof}
\end{lemma}

Our goal is to show that there is $k$ such that the signs in the sequence
\begin{equation*}
    \hat{w}^k_1 \hat{D}^k_1, \hat{w}^k_2 \hat{D}^k_2, \dots, \hat{w}^k_{r+1} \hat{D}^k_{r+1}
\end{equation*}
alternate. It is equivalent to the property $\sign {\hat{w}^k_i \hat{D}^k_i \hat{w}^k_j \hat{D}^k_j} = (-1)^{i-j}$ for all $i$ and $j$. The next lemma shows that if this property is not satisfied for some $k$ and a pair of positions $(i, j)$, then it is necessarily satisfied for $i$ or $j$ (or even both) and another pair of positions.

\begin{lemma}
\label{lemma:lemma_table_properties}
Let the signs in the sequence
\begin{equation*}
w_1 D_1, w_2 D_2, \dots, w_{r+1} D_{r+1}
\end{equation*}
alternate. Let
\begin{equation*}
\sign(\hat{w}_i^k \hat{D}_i^k \hat{w}_j^k \hat{D}_j^k) = (-1)^{i-j+1}.
\end{equation*}
Then
\begin{itemize}
\item if $i=k$, then
$\sign(\hat{w}_i^j \hat{D}_i^j \hat{w}_j^j \hat{D}_j^j) = (-1)^{i-j}$;
\item if $j=k$, then
$\sign(\hat{w}_j^i \hat{D}_j^i \hat{w}_i^i \hat{D}_i^i) = (-1)^{i-j}$;
\item if $i \neq k$ and $j \neq k$, then
$\sign(\hat{w}_k^i \hat{D}_k^i \hat{w}_j^i \hat{D}_j^i) = (-1)^{k-j}$
and
$\sign(\hat{w}_k^j \hat{D}_k^j \hat{w}_i^j \hat{D}_i^j) = (-1)^{k-i}$.
\end{itemize}
\begin{proof}
Let $i=j$. Then the condition of the lemma cannot be fulfilled, since
\begin{equation*}
\sign(\hat{w}_i^k \hat{D}_i^k \hat{w}_i^k \hat{D}_i^k) = 1 \neq (-1)^{i-i+1} = -1.
\end{equation*}

Let $i=k \neq j$. Then the condition of the lemma can be written as
\begin{equation*}
    \sign(\hat{w}_k^k \hat{D}_k^k \hat{w}_j^k \hat{D}_j^k) = (-1)^{k-j+1},
\end{equation*}
which by the definitions of $\hat{D}_k^k$ and $\hat{w}^k$ can be written as
\begin{equation}
\label{eq:trr1}
    \sign(w_{r+2} D_k w_j \hat{D}_j^k) = (-1)^{k-j+1}.
\end{equation}
Since the signs in the sequence $w_1 D_1, w_2 D_2, \dots, w_{r+1} D_{r+1}$
alternate,
\begin{equation}
\label{eq:trr2}
    \sign(w_k D_k w_j D_j) = (-1)^{k-j}.
\end{equation}
Let us multiply \eq{trr1} and \eq{trr2}. Then
\begin{multline*}
\sign(w_{r+2} w_k D_j \hat{D}_j^k) = \sign(w_k D_k w_j D_j) \sign(w_{r+2} D_k w_j \hat{D}_j^k) = \\
(-1)^{k-j} (-1)^{k-j+1} = -1.
\end{multline*}
Since $\hat{D}_k^j = (-1)^{k-j+1} \hat{D}_j^k$, we have that
\begin{equation*}
\sign(w_{r+2} D_j w_k \hat{D}_k^j) = (-1) (-1)^{k-j+1} = (-1)^{k-j}.
\end{equation*}
It remains to note that $\hat{w}_j^j = w_{r+2}$, $\hat{w}_k^j = w_k$ and $\hat{D}_j^j = D_j$, whence
\begin{equation*}
    \sign(\hat{w}_j^j \hat{D}_j^j \hat{w}_k^j \hat{D}_k^j) = (-1)^{k-j}.
\end{equation*}
Thus, the first part of the lemma is proven. The second part can be obtained from the first one by rearranging the factors.

Let us prove the third part. Let $i$, $j$ and $k$ be pairwise distinct. Then the condition of the lemma can be written as
\begin{equation}
\label{eq:smth1}
\sign(w_i \hat{D}_i^k w_j \hat{D}_j^k) = (-1)^{i-j+1}.
\end{equation}
Let us prove that in this case $\sign(w_k \hat{D}_k^i w_j \hat{D}_j^i) = (-1)^{k-j}$.
On the contrary, let
\begin{equation}
\label{eq:smth2}
\sign(w_k \hat{D}_k^i w_j \hat{D}_j^i) = (-1)^{k-j+1}.
\end{equation}
Multiplying \eq{smth1} by the alternance condition $\sign(w_i D_i w_j D_j) = (-1)^{i-j}$,
we get that
\begin{equation}
\label{eq:contr_1}
\sign(\hat{D}_i^k D_i \hat{D}_j^k D_j) = \sign(w_i \hat{D}_i^k w_j \hat{D}_j^k) \sign(w_i D_i w_j D_j) = (-1)^{i-j+1}(-1)^{i-j} = -1.
\end{equation}
Similarly, multiplying \eq{smth2} by $\sign(w_k D_k w_j D_j) = (-1)^{k-j}$
we get that
\begin{equation}
\label{eq:contr_2}
\sign(\hat{D}_k^i D_k \hat{D}_j^i D_j) = \sign(w_k \hat{D}_k^i w_j \hat{D}_j^i) \sign(w_k D_k w_j D_j) = (-1)^{k-j+1}(-1)^{k-j} = -1.
\end{equation}
It remains to note that \eq{contr_1} and \eq{contr_2} contradict \lem{determinants_lemma}, whence
\begin{equation}
\label{eq:smth3}
\sign(w_k \hat{D}_k^i w_j \hat{D}_j^i) = (-1)^{k-j}.
\end{equation}
After replacing $w_k$ by $\hat{w}_k^i$ and $w_j$ by $\hat{w}_j^i$ using the corresponding definition, we obtain that \eq{smth3} becomes the first part of the third statement of the lemma. The second part follows from the first one, since $i$ and $j$ are symmetric in the condition and statement of the lemma.
\end{proof}
\end{lemma}

The next lemma shows that if the first condition of \lem{lemma_table_properties} is met, then its statement leads to the fact that there is always some row such that the signs alternate.

\begin{lemma}
\label{lemma:table_alternating_signs}
Let a matrix $S \in \mathbb{R}^{n \times n}$ is such that $s_{ki} \in \{-1, 1\}$ and satisfies the following property: if for a triple $(i, j, k)$, where $1 \le i, j, k \le n$, we have $s_{ki} s_{kj} = (-1)^{i-j+1}$, then
\begin{itemize}
\item if $i=k$, then $s_{jk} s_{jj} = (-1)^{k-j}$;
\item if $j=k$, then $s_{ik} s_{ii} = (-1)^{k-i}$;
\item if $i$, $j$, and $k$ are pairwise distinct, then
$s_{ik} s_{ij} = (-1)^{k-j}$
and
$s_{jk} s_{ji} = (-1)^{k-i}$.
\end{itemize}
Then the matrix $S$ has a row with the alternating signs.
\begin{proof}
Let us prove the lemma by induction. Let $n=2$. If $s_{11} s_{12} = -1$, then the first row has the alternating signs. Otherwise, $s_{11} s_{12} = 1$ and the condition of the property in the lemma is fulfilled for $i=1$, $j=2$, $k=1$. Then
$s_{21} s_{22} = -1$
and the second row has alternating signs.

Let us assume the statement is true for $n-1$ and prove it for the matrix of size $n$. If the condition of the lemma is fulfilled for a matrix $S$, it is also satisfied for its principal submatrix $\hat{S}$. By the induction hypothesis, $\hat{S}$ has a row with the number $t$ such that its elements have alternating signs. If $s_{t, n-1} s_{t, n} = -1$, then the signs alternate in the row number $t$ in the matrix $S$. Let $s_{t, n-1} s_{t, n} = 1$. Then, due to the alternation of signs in the row number $t$, we have $s_{t, i} s_{t, n} = (-1)^{i-n+1}$ for $i=1,\dots,n$ and $j=n$.
Then from the condition of the lemma for $i=t$ we have $s_{n,t} s_{n, n} = (-1)^{t-n}$ (first case)
and for $i\neq t$ we have $s_{n, t} s_{n, i} = (-1)^{t-i}$ (second case),
so
\begin{equation*}
s_{n, t} s_{n, i} = (-1)^{t-i}, \quad i=1,\dots,n,
\end{equation*}
which means the alternance of signs in the last row.
\end{proof}
\end{lemma}

\begin{proof}[Proof of \lem{interchange_sign_alternate}]
Let $\hat{V} \in \mathbb{R}^{(r+1)\times r}$ and $h \in \mathbb{R}^r$ are such that $V = \begin{bmatrix}
    \hat{V}\\
    h
\end{bmatrix}$ is Chebyshev. Let also $\hat{w} \in \mathbb{R}^{r+1}$ be a vector with non-zero components and $\xi \in \mathbb{R}$, $\xi \neq 0$. We denote $w = \begin{bmatrix}
    \hat{w}^T & \xi
\end{bmatrix}^T$. In the notations of \sect{appendix_proof} we define the matrix $S \in \mathbb{R}^{(r+1)\times (r+1)}$ such that $s_{ki} = \sign (\hat{w}_i^k \hat{D}_i^k)$. It is easy to see that the first condition of \lem{lemma_table_properties} is met, from which the property of \lem{table_alternating_signs} follows, which in turn implies the statement of \lem{interchange_sign_alternate}.
\end{proof}

\end{document}